\documentclass[10pt]{article}

\usepackage{array}
\usepackage{multirow}

\makeatletter
\newcommand{\thickhline}{%
    \noalign {\ifnum 0=`}\fi \hrule height 1pt
    \futurelet \reserved@a \@xhline
}
\newcolumntype{"}{@{\hskip\tabcolsep\vrule width 1pt\hskip\tabcolsep}}
\makeatother

\usepackage{graphics}
\usepackage{epsfig}
\usepackage{subfigure}
\usepackage{amssymb}
\usepackage{amsthm}
\usepackage{mathrsfs}
\usepackage{hhline}
\usepackage{longtable}
\usepackage{amsmath}
\usepackage{array}
\usepackage{float}
\usepackage{picinpar}
\usepackage{multirow}
\usepackage{cite}
\usepackage{enumerate}
\usepackage{xcolor}
\usepackage[mathlines]{lineno}
\modulolinenumbers[2]
\newcommand*\patchAmsMathEnvironmentForLineno[1]{%
\expandafter\let\csname old#1\expandafter\endcsname\csname #1\endcsname  \expandafter\let\csname oldend#1\expandafter\endcsname\csname end#1\endcsname  \renewenvironment{#1}%
{\linenomath\csname old#1\endcsname}%
{\csname oldend#1\endcsname\endlinenomath}}%
\newcommand*\patchBothAmsMathEnvironmentsForLineno[1]{%
\patchAmsMathEnvironmentForLineno{#1}%
\patchAmsMathEnvironmentForLineno{#1*}}%
\AtBeginDocument{%
\patchBothAmsMathEnvironmentsForLineno{equation}%
\patchBothAmsMathEnvironmentsForLineno{align}%
\patchBothAmsMathEnvironmentsForLineno{flalign}%
\patchBothAmsMathEnvironmentsForLineno{alignat}%
\patchBothAmsMathEnvironmentsForLineno{gather}%
\patchBothAmsMathEnvironmentsForLineno{multline}%
}

\newtheorem{theorem}{Theorem}[section]
\newtheorem{lemma}[theorem]{Lemma}
\newtheorem{example}{Example}[section]
\newtheorem{corollary}[theorem]{Corollary}
\newtheorem{problem}{Problem}[section]

\newtheorem{conjecture}{Conjecture}[section]

\numberwithin{equation}{section}
\def\Z{\mathbb{Z}}

\newtheorem{defi}{Definition}[section]

\usepackage{graphicx}
\graphicspath{%
    {converted_graphics/}
    {/}
}
\begin{document}
\baselineskip18truept
\normalsize
\begin{center}
{\mathversion{bold}\Large \bf On local antimagic chromatic number of cycle-related join graphs}

\bigskip
{\large  Gee-Choon Lau{$^{a,}$}\footnote{Corresponding author.}, Wai-Chee Shiu{$^b$}, Ho-Kuen Ng{$^c$}}\\

\medskip

\emph{{$^a$}Faculty of Computer \& Mathematical Sciences,}\\
\emph{Universiti Teknologi MARA (Segamat Campus),}\\
\emph{85000, Johor, Malaysia.}\\
\emph{geeclau@yahoo.com}\\

\medskip

\emph{{$^b$}Department of Mathematics, Hong Kong Baptist University,}\\
\emph{224 Waterloo Road, Kowloon Tong, Hong Kong, P.R. China.}\\
\emph{wcshiu@hkbu.edu.hk}\\

\medskip

\emph{{$^c$}Department of Mathematics, San Jos\'{e} State University,}\\
\emph{San Jos\'e CA 95192 USA.}\\
\emph{ho-kuen.ng@sjsu.edu}\\
\end{center}


\medskip
\begin{abstract}
An edge labeling of a connected graph $G = (V, E)$ is said to be local antimagic if it is a bijection $f:E \to\{1,\ldots ,|E|\}$ such that for any pair of adjacent vertices $x$ and $y$, $f^+(x)\not= f^+(y)$, where the induced vertex label $f^+(x)= \sum f(e)$, with $e$ ranging over all the edges incident to $x$.  The local antimagic chromatic number of $G$, denoted by $\chi_{la}(G)$, is the minimum number of distinct induced vertex labels over all local antimagic labelings of $G$. In this paper, several sufficient conditions for $\chi_{la}(H)\le \chi_{la}(G)$ are obtained, where $H$ is obtained from $G$ with a certain edge deleted or added. We then determined the exact value of the local antimagic chromatic number of many cycle related join graphs.

\medskip

\end{abstract}

\tolerance=10000
\baselineskip12truept
\def\qed{\hspace*{\fill}$\Box$\medskip}

\def\s{\,\,\,}
\def\ss{\smallskip}
\def\ms{\medskip}
\def\bs{\bigskip}
\def\c{\centerline}
\def\nt{\noindent}
\def\ul{\underline}
\def\lc{\lceil}
\def\rc{\rceil}
\def\lf{\lfloor}
\def\rf{\rfloor}
\def\a{\alpha}
\def\b{\beta}
\def\n{\nu}
\def\o{\omega}
\def\ov{\over}
\def\m{\mu}
\def\t{\tau}
\def\th{\theta}
\def\k{\kappa}
\def\l{\lambda}
\def\L{\Lambda}
\def\g{\gamma}
\def\d{\delta}
\def\D{\Delta}
\def\e{\epsilon}
\def\lg{\langle}
\def\rg{\tongle}
\def\p{\prime}
\def\sg{\sigma}
\def\to{\rightarrow}

\newcommand{\K}{K\lower0.2cm\hbox{4}\ }
\newcommand{\cl}{\centerline}
\newcommand{\om}{\omega}
\newcommand{\ben}{\begin{enumerate}}

\newcommand{\een}{\end{enumerate}}
\newcommand{\bit}{\begin{itemize}}
\newcommand{\eit}{\end{itemize}}
\newcommand{\bea}{\begin{eqnarray*}}
\newcommand{\eea}{\end{eqnarray*}}
\newcommand{\bear}{\begin{eqnarray}}
\newcommand{\eear}{\end{eqnarray}}

\section{Introduction}

\nt A connected graph $G = (V, E)$ is said to be {\it local antimagic} if it admits a {\it local antimagic edge labeling}, i.e., a bijection $f : E \to \{1,\dots ,|E|\}$ such that the induced vertex labeling $f^+ : V \to \Z$ given by $f^+(u) = \sum f(e)$ (with $e$ ranging over all the edges incident to $u$) has the property that any two adjacent vertices have distinct induced vertex labels (see~\cite{Arumugam, Bensmail}). Thus, $f^+$ is a coloring of $G$. Clearly, the order of $G$ must be at least 3.  The vertex label $f^+(u)$ is called the {\it induced color} of $u$ under $f$ (the {\it color} of $u$, for short, if no ambiguity occurs). The number of distinct induced colors under $f$ is denoted by $c(f)$, and is called the {\it color number} of $f$. The {\it local antimagic chromatic number} of $G$, denoted by $\chi_{la}(G)$, is $\min\{c(f) : f\mbox{ is a local antimagic labeling of } G\}$.

\ms\nt Let $O_n = \overline{K_n}$ be the empty graph of order $n\ge 1$. For any graph $G$, the join graph $H=G \vee O_n$ is defined by $V(H)=V(G)\cup\{v_j : 1\le j\le n\}$ and $E(H)=E(G)\cup\{uv_j : u \in V(G), 1\le j\le n\}$.  In~\cite[Theorem 2.16]{Arumugam}, it was claimed that for any $G$ with order $n\ge4$, $$\chi_{la}(G) + 1 \le \chi_{la}(G\vee O_2) \le \begin{cases}
\chi_{la}(G)+1 & \mbox{if } n \mbox{ is even,} \\
\chi_{la}(G) + 2 & \mbox{if } n \mbox{ is odd}.
\end{cases}$$
In~\cite{LNS}, Lau {\it et al.} showed that there exists a graph $G$ of order $n$ such that (i) $\chi_{la}(G) - \chi_{la}(G\vee O_2) = n-3$ for each even $n\ge 4$, and (ii) $\chi_{la}(G) = \chi_{la}(G\vee O_2)$ for each odd $n\ge 3$. This implies that the above lower bound is invalid. They then showed that $\chi_{la}(G + O_n)\ge \chi(G) + 1$ and the bound is sharp. Several sufficient conditions for the following conjecture to hold were also given.

\begin{conjecture}\label{conj-joint} For $n\ge 1$, $\chi_{la}(G\vee O_n)\ge\chi_{la}(G)+1$ if and only if $\chi(G)=\chi_{la}(G)$.  \end{conjecture}

\nt Let $G-e$ (or $G+e$) be the graph $G$ with an edge $e$ deleted (or added). As a natural extension, we have in this paper obtained several sufficient conditions for $\chi_{la}(G-e)\le \chi_{la}(G)$ (or $\chi_{la}(G+e)\le \chi_{la}(G)$). We then determine the exact value of the local antimagic chromatic number of many cycle related join graphs. We shall use the notation $[a,b]=\{c\in\Z : a\le c\le b\}$, for integers $a\le b$. Unless stated otherwise, all graphs considered in this paper are simple, undirected, connected and of order at least 3. Thus $\chi_{la}(G)\ge 2$ for any graph $G$. Interested readers may refer to Yu~\cite{Yu} for local antimagic labeling of subcubic graphs without isolated edges.


\ms\nt For $m,n\ge 2$, it is well known that a magic $(m,n)$-rectangle exists if and only if $m\equiv n\pmod{2}$ and $(m,n)\not=(2,2)$ (see \cite{Chai, Reyes}). Let $a_{i,j}$ be the $(i,j)$-entry of a magic $(m,n)$-rectangle with row constant $n(mn+1)/2$ and column constant $m(mn+1)/2$.

\section{Bounds on graphs with an edge deleted or added}\label{sec:edge}

\nt Observe that $K_t$, $t\ge 3$, is a complete $t$-partite graph with $\chi_{la}(K_t) = t$. The contrapositive of the following lemma gives a sufficient condition for a bipartite graph $G$ to have $\chi_{la}(G)\ge 3$.

\begin{lemma}\label{lem-2part} Let $G$ be a graph of size $q$. Suppose there is a local antimagic labeling of $G$ inducing a $2$-coloring of $G$ with colors $x$ and $y$, where $x<y$. Let $X$ and $Y$ be the numbers of vertices of colors $x$ and $y$, respectively. Then $G$ is a bipartite graph whose sizes of parts are $X$ and $Y$ with $X>Y$, and
\begin{equation}\label{eq-bi} xX=yY= \frac{q(q+1)}{2}.\end{equation}
\end{lemma}

\begin{proof} Clearly $G$ is bipartite. Each edge is incident with one vertex of color $x$ and one vertex of color $y$. Hence we have the equation~\eqref{eq-bi}. Since $x< y$, $X>Y$. This completes the proof.
\end{proof}

\begin{lemma}\label{lem-reg} Suppose $G$ is a $d$-regular graph of size $q$. If $f$ is a local antimagic labeling of $G$, then $g = q + 1 - f$ is also a local antimagic labeling of $G$ with $c(f)= c(g)$.  Moreover, suppose $c(f)=\chi_{la}(G)$ and if $f(uv)=1$ or $f(uv)=q$, then $\chi_{la}(G-uv)\le \chi_{la}(G)$. \end{lemma}

\begin{proof} Let $x,y\in V(G)$. Here, $g^+(x) = d(e+1) - f^+(x)$ and $g^+(y) = d(q+1) - f^+(y)$. Therefore, $f^+(x)=f^+(y)$ if and only if $g^+(x) = g^+(y)$. Thus, $g$ is also a local antimagic labeling of $G$ with $c(g) = c(f)$.

\ms\nt If $f(uv)=q$, then we may consider $g=q+1-f$. So without loss of generality, we may assume that $f(uv)=1$. Define $h : E(G-uv) \to [1, |E(G)|-1]$ such that $h(e) = f(e) - 1$ for $e\ne uv$. So, $h^+(x) = f^+(x) - d$ for each vertex $x$ of $G-uv$. Therefore, $f^+(x)=f^+(y)$ if and only if $h^+(x) = h^+(y)$. Thus, $h$ is also a local antimagic labeling of $G$ with $c(h) = c(f)$. Consequently, $\chi_{la}(G-uv)\le \chi_{la}(G)$. \end{proof}

\nt Note that if $G$ is a regular edge-transitive graph, then $\chi_{la}(G-e)\le \chi_{la}(G)$. 


\begin{lemma}\label{lem-nonreg} Suppose $G$ is a graph of size $q$ and $f$ is a local antimagic labeling of $G$. For any $x,y\in V(G)$, if\\
(i) $f^+(x) = f^+(y)$ implies that $\deg(x)=\deg(y)$, and\\
(ii) $f^+(x) \ne f^+(y)$ implies that $(q+1)(\deg(x)-\deg(y)) \ne f^+(x) - f^+(y)$,\\ then $g = q + 1 - f$ is also a local antimagic labeling of $G$ with $c(f)= c(g)$. \end{lemma}

\begin{proof} For any $x,y\in V(G)$, we have $g^+(x) = \deg(x)(q+1) - f^+(x)$ and $g^+(y) = \deg(y)(q+1) - f^+(y)$. Here $g^+(x)-g^+(y)=(q+1)(\deg(x)-\deg(y))-(f^+(x)-f^+(y))$. If $f^+(x)=f^+(y)$, then condition (i) implies that $g^+(x) = g^+(y)$. If $f^+(x) \ne f^+(y)$, then condition (ii) implies that $g^+(x) \ne g^+(y)$. Thus, $g$ is also a local antimagic labeling of $G$ with $c(g) = c(f)$.

\end{proof}  

\nt For $t\ge 2$, consider the following conditions for a graph $G$:

\begin{enumerate}[(i)]
  \item $\chi_{la}(G)=t$ and $f$ is a local antimagic labeling of $G$ that induces a $t$-independent partition $\cup^t_{i=1} V_i$ of $V(G)$.
  \item For each $x\in V_k$, $1\le k\le t$, $deg(x)=d_k$ satisfying $f^+(x) - d_a \ne f^+(y) - d_b$, where $x\in V_a$ and $y\in V_b$ for $1\le a\ne b\le t$. 
  \item There exist two non-adjacent vertices $u,v$ with $u\in V_i, v\in V_j$ for some $1\le i\ne j\le t$ such that 
  \begin{enumerate}[(a)]
    \item $|V_i|=|V_j|=1$ and $\deg(x)=d_k$ for $x\in V_k$, $1\le k\le t$; or
    \item $|V_i|=1$, $|V_j|\ge 2$ and $\deg(x)=d_k$ for $x\in V_k$, $1\le k\le t$ except that $\deg(v) = d_j-1$; or
    \item $|V_i|\ge 2$, $|V_j|\ge 2$ and $deg(x)=d_k$ for $x\in V_k$, $1\le k\le t$ except that $\deg(u) = d_i-1$, $\deg(v) = d_j-1$,
\end{enumerate}
each satisfying $f^+(x) + d_a\ne f^+(y) + d_b$, where $x\in V_a$ and $y\in V_b$ for $1\le a\ne b\le t$.
\end{enumerate}

\begin{lemma}\label{lem-G-e} Let $H$ be obtained from $G$ with an edge $e$ deleted. If $G$ satisfies Conditions (i) and (ii) and $f(e)=1$, then $\chi(H)\le \chi_{la}(H)\le t$. \end{lemma}

\begin{proof} By definition, we have the lower bound. Define $g : E(H)\to [1, |E(H)|]$ such that $g(e') = f(e') - 1$ for each $e'\in E(H)$. Observe that $g$ is a bijection with $g^+(x) = f^+(x) - d_k$ for each $x\in V_k$, $1\le k\le t$. Thus, $g^+(x)=g^+(y)$ if and only if $x,y\in V_k$, $1\le k\le t$. Therefore, $g$ is a local antimagic labeling of $H$ with $c(g) = c(f)$. Thus, $\chi_{la}(H)\le t$.
\end{proof}

\begin{lemma}\label{lem-G+e} Suppose $uv\not\in E(G)$. Let $H$ be obtained from $G$ with an edge $uv$ added. If $G$ satisfies Conditions (i) and (iii), then $\chi(H)\le \chi_{la}(H)\le t$.
\end{lemma}

\begin{proof} By definition, we have the lower bound. Define $g: E(H) \to [1,|E(H)|]$ such that $g(uv) = 1$ and $g(e) = f(e) + 1$ for $e\in E(G)$. Observe that $g$ is a bijection with $g^+(x) = f^+(x) + d_k$ for each $x\in V_k$, $1\le k\le t$. Thus, $g^+(x)=g^+(y)$ if and only if $x,y\in V_k$, $1\le k\le t$. Therefore, $g$ is a local antimagic labeling of $H$ with $c(g) = c(f)$. Thus, $\chi_{la}(H)\le t$.
\end{proof}

\nt In~\cite[Theorem 2.11]{Arumugam}, the authors showed that for any two distinct integers $m,n\ge 2$, $\chi_{la}(K_{m,n}) = 2$ if and only if $m\equiv n\pmod{2}$. Let $K^-_{m,n}$ be the graph $K_{m,n}$ with an edge deleted.  From the proof of ~\cite[Theorem 2.11]{Arumugam} and by Lemma~\ref{lem-G-e}, the following result is obvious. 

\begin{corollary} For any two distinct integers $m,n\ge 2$ and $m\equiv n\pmod{2}$, $\chi_{la}(K^-_{m,n}) = 2$. \end{corollary}

\section{Cycle-related join graphs}\label{sec:join}

\nt Consider the join graph $C_m\vee O_n$ with $V(C_m)=\{u_i : 1\le  i\le m\}$, $V(O_n) = \{v_j : 1\le j\le n\}$ and $E(C_m\vee O_n)=\{u_iu_{i+1}: 1\le i\le m\} \cup \{u_iv_j : 1\le i\le m, 1\le j\le n\}$, where $u_{m+1}=u_1$. Let $e_i=u_iu_{i+1}$ for $1\le i\le m$. So $e_m=u_mu_1$. We shall keep these notations in this section unless stated otherwise.

\begin{theorem}\label{thm-CmOnodd}  For odd $m,n\ge 3$, $\chi_{la}(C_m \vee O_n) = 4$.      \end{theorem}

\begin{proof}  Define an edge labeling $f : E(C_m\vee O_n) \to [1, mn+m]$ such that $f(e_{2i-1}) = i$ $(1\le i\le (m+1)/2)$ and $f(e_{2i}) = m+1-i$ $(1\le i\le (m-1)/2)$ and that $f(u_iv_j)$ is the $(i,j)$-entry of a magic $(m,n)$-rectangle containing integers in $[m+1,mn+m]$ with row sum constant $n(mn+1)/2 + mn$ and column sum constant $m(mn+1)/2 + m^2$. One can check that

\begin{enumerate}[(i)]
  \item $f^+(v_j)=m(mn+1)/2 + m^2$,
  \item $f^+(u_1)=n(mn+1)/2 + mn + (m+3)/2$,
  \item $f^+(u_i) = n(mn+1)/2 + mn + m+1$ for even $i$, and
  \item $f^+(u_i) = n(mn+1)/2 + mn + m+2$ for odd $i\ge 3$.
\end{enumerate}

\nt Suppose $m\le n$. We have $m(mn+1)/2 + m^2 < n(mn+1)/2 + mn + (m+3)/2 < n(mn+1)/2 + mn + m+1 < n(mn+1)/2 + mn + m+2$. So, $\chi_{la}(G) \le 4$.

\ms\nt Suppose $m > n$. We have $m(mn+1)/2 + m^2 = n(mn+1)/2 + mn + (m-n)m + (m-n)(mn+1)/2 > n(mn+1)/2 + mn + m+2$. So, $\chi_{la}(G) \le 4$.

\ms\nt Since $\chi_{la}(G) \ge \chi(G) = 4$, we have $\chi_{la}(G) = 4$.   \end{proof}  %

\begin{corollary} For odd $m,n\ge 3$, if $H=(C_m\vee O_n)-e$ where $e\not\in E(C_m)$, then $\chi_{la}(H) = 4$.
\end{corollary}

\begin{proof} Note that $G=C_m\vee O_n$ has size $mn+m$ and every vertex belonging to $C_m$ (or $O_n$) has degree $n+2$ (or $m$). Let $f$ be the local antimagic labeling as defined in the proof of Theorem~\ref{thm-CmOnodd}. We can check that $f$ satisfies the conditions of Lemma~\ref{lem-nonreg}. Therefore, $g = mn + m + 1 - f$ is also a local antimagic labeling of $G$ with $c(g)=4$ such that $g(e)=1$ for an edge $e\not\in E(C_m)$. It is straightforward to check the conditions of Lemma~\ref{lem-G-e}. By Lemma~\ref{lem-G-e}, we have $4=\chi(H) \le \chi_{la}(H)\le 4$. Thus, the result holds.
\end{proof}

\begin{theorem}\label{thm-CmOneven} For $m\ge 2$ and $n\ge 1$, $\chi_{la}(C_{2m} \vee O_{2n}) = 3$. \end{theorem}

\begin{proof} Let $G=C_{2m}\vee O_{2n}$. Define an edge labeling $f : E(G) \to [1,4mn+2m]$ such that $f(e_h)=h$ for $1\le h\le 2m$ and $f(u_hv_k)$ is given below, for $1\le h\le 2m$ and $1\le k\le 2n$.

\ms\nt We define $f(u_1v_1)=2m+1$ and $f(u_{2i-1}v_{1})=4m-2i+3$ for $2\le i\le m$. For $1\le i\le m$, define
\begin{enumerate}[(i)]
  \item $f(u_{2i-1}v_{2})=6m-2i+1$;
  \item $f(u_{2i-1}v_{2j-1})=2m(j-1)+2i$ and $f(u_{2i-1}v_{2j})=2m(2n+1-j)-2i+2$, for $2\le j\le n$,
  \item $f(u_{2i}v_{1})=2m(2n+1)-2i+2$ and $f(u_{2i}v_{2})=4mn-2i+2$,
  \item $f(u_{2i}v_{2j-1})=2m(2n-j+3)-2i+1$ and $f(u_{2i}v_{2j})=2m(j+1)+2i-1$, for $2\le j\le n$.
\end{enumerate}
\nt One may check that $f$ is a bijection. Observe that
\begin{enumerate}[(i)]
  \item $f(u_{2i-1}v_{1})+f(u_{2i-1}v_{2})=10m-4i+4$ and $f(u_{2i}v_{1})+f(u_{2i}v_{2})=8mn+2m-4i+4$ for $1\le i\le m$,
  \item $f(u_{2i}v_{2j-1})+f(u_{2i}v_{2j})=4m(n+2)$ for $1\le i\le m$ and $2\le j\le n$,
  \item $f(u_{2i-1}v_{2j-1})+f(u_{2i-1}v_{2j})=4mn+2$ for $1\le i\le m$ and $2\le j\le n$.
\end{enumerate}

\nt Thus
\begin{align*}
f^+(u_1) & =f(e_1)+f(e_{2m})+f(u_1v_1)+f(u_1v_2)+\sum\limits_{j=2}^n (4mn+2)= 4mn^2-4mn+2n+10m-1;\\
f^+(u_{2i-1}) &=f(e_{2i-2})+f(e_{2i-1})+(10m-4i+4)+\sum\limits_{j=2}^n (4mn+2)\\
& = (4i-3)+(10m-4i+4)+(4mn+2)(n-1)= 4mn^2-4mn+2n+10m-1
\mbox{ if $2\le i\le m$};\\
f^+(u_{2i}) &=f(e_{2i-1})+f(e_{2i})+(8mn+2m-4i+4)+\sum\limits_{j=2}^n 4m(n+2)\\
& = (8mn+2m+3)+4m(n+2)(n-1)= 4mn^2+12mn-6m+3
\mbox{ if $1\le i\le m$},
\end{align*}
\begin{align*}
f^+(v_1) & = (2m+1)+\sum_{i=2}^m(4m-2i+3) +\sum_{i=1}^m(4mn+2m-2i+2)= 4m^2n+4m^2+m;\\
f^+(v_2) & = \sum_{i=1}^m (4mn+6m-4i+3)= 4m^2n+4m^2+m;\\
f^+(v_k) & = \sum_{i=1}^m (4mn +4m+1)= 4m^2n+4m^2+m \mbox{ if $3\le k\le 2n$}.
\end{align*}

\nt Now, let $g_1 = f^+(u_{2i-1}) = 4mn^2-4mn+2n+10m-1$, $g_2 = f^+(u_{2i}) = 4mn^2+12mn-6m+3$, and $g_3 = f^+(v_j) =4m^2n+4m^2+m$. Clearly, $g_1 < g_2$.

\ms\nt Suppose $n\ge m$. We have $g_2-g_3=4mn(n-m)+m(12n-4m-7)+6>0$. Suppose $m > n$. $g_3-g_2=4mn(m-n-2)+m(4m-4n+7)-3$. When $m-n\ge 2$, clearly $g_3>g_2$. For $m-n=1$, $g_3-g_2=-4m^2+15m-3\ne 0$.

\ms\nt We now consider $g_3 - g_1 =2n[2m(m-n)-1]+m(4n+4m-9)+1$. If $m \ge n$, then $g_3 - g_1\ge 2n(m-1)+m(2n+4m-9)+1>0$.
Suppose $n > m$. Now $g_1 - g_3 =4mn(n-m-2)+4m(n-m)+2n+9m-1>0$ when $n-m\ge 2$. When $n-m=1$. $g_1 - g_3=-4m^2+11m+1\ne 0$.

\ms\nt Thus, $\chi_{la}(G) \le 3$. Since $\chi_{la}(G) \ge \chi(G) = 3$, we have $\chi_{la}(G) = 3$.   \end{proof}

\begin{corollary} For $m\ge 2$, $n\ge 1$, if $H=(C_{2m}\vee O_{2n})-e$, then $\chi_{la}(H) = 3$, where $e$ is an edge of $C_{2m}\vee O_{2n}$.
\end{corollary}

\begin{proof} Note that $G=C_{2m}\vee O_{2n}$ has size $4mn+2m$ where every vertex belonging to  $C_{2m}$ (or $O_{2n}$) has degree $2n+2$ (or $2m$). Let $f$ be the local antimagic labeling as defined in the proof of Theorem~\ref{thm-CmOneven}. Suppose $e\in E(C_{2m})$. It is straightforward to check that $f$ satisfies the conditions of Lemma~\ref{lem-G-e}. Thus, we have $3=\chi(H) \le \chi_{la}(H)\le 3$.  Suppose $e\not\in E(C_{2m})$. We can check that $f$ satisfies the conditions of Lemma~\ref{lem-nonreg}. Therefore, $g = 4mn + 2m + 1 - f$ is also a local antimagic labeling of $G$ with $c(g)=3$ such that $g(e)=1$. It is straightforward to check the conditions of Lemma~\ref{lem-G-e}. By Lemma~\ref{lem-G-e}, we have $3=\chi(H) \le \chi_{la}(H)\le 3$. Thus, the result holds.
\end{proof}

\nt Since for odd $m,n\ge 3$, $\chi_{la}(C_m\vee O_n) = \chi_{la}(C_m) + 1 = \chi(C_m) + 1$, and for even $n\ge 2$, $\chi_{la}(C_m \vee O_n) = \chi_{la}(C_m) = \chi(C_m) + 1$, Theorems~\ref{thm-CmOnodd} and~\ref{thm-CmOneven} provide further evidence that Conjecture~\ref{conj-joint} holds.

\ms\nt Note that $C_m \vee O_1 = W_m$, the wheel graph of order $m+1\ge 4$. In~\cite[Theorem 6]{LNS}, the authors proved that $\chi_{la}(W_m)=3$ if $m\equiv0\pmod{4}$. In~\cite[Theorem 2.14]{Arumugam}, the authors proved that $\chi_{la}(W_m)=3$ if $m\equiv2\pmod{4}$, and $\chi_{la}(W_m)=4$ if $m$ is odd. We note that for $m\equiv 1\pmod{4}$, the defined local antimagic labeling $f$ (or $f_3$ in the proof) has three errors that should be corrected as $f(v_iv) = (8m + 5 - i)/4$ for $i\equiv 1\pmod{4}, i\ne 1$; $f(v_iv) = (7m + 4 - i)/4$ for $i\equiv 3\pmod{4}$; and $f^+(v_i) = (11m+13)/4$ for odd $i\ne 1$. Moreover, for $m\equiv 3\pmod{4}$, the induced vertex label for $v_i$, $i\ne 1$ is odd, should be $9(m+1)/4$.  


\begin{theorem}\label{thm-W4-e} $$\chi_{la}(W_4 - e) =\begin{cases} 3 & \mbox{ if } e\not\in E(C_4),\\ 4 & \mbox{ otherwise.} \end{cases}$$  \end{theorem}

\begin{proof} The graph in Figure~\ref{fig:G4-1} shows that $W_4-e$ admits a local antimagic labeling $f$ with $c(f)=3$ so that $\chi_{la}(W_4-e)=3$ if $e\not\in E(C_4)$.

\begin{figure}[H] 
  \centering
  \includegraphics[bb=0 0 213 216,width=1in,height=1in,keepaspectratio]{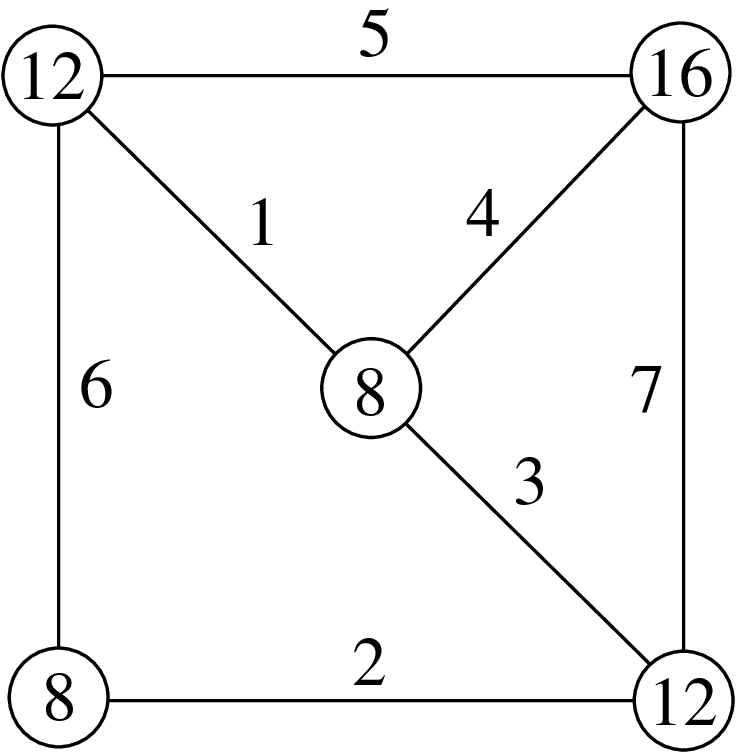}
  \caption{$W_4 - e$}
  \label{fig:G4-1}
\end{figure}

\ms\nt Suppose $e\in E(C_4)$. Without loss of generality we may assume that $e=u_4u_1$. Suppose there were a local antimagic labeling $f$ of $W_4-e$ with $c(f)=3$. Then $f^+(v_1)=c$, $f^+(u_1)=f^+(u_3)=a$ and $f^+(u_2)=f^+(u_4)=b$, where $a,b,c$ are distinct.

\ms\nt Clearly
\begin{equation}\label{eq-W4} 28=\sum_{i=1}^7 i =2a+f(v_1u_2)+f(v_1u_4)=2b+f(v_1u_1)+f(v_1u_3).\end{equation}
Thus, $f(v_1u_2)\equiv f(v_1u_4)\pmod{2}$ and $f(v_1u_1)\equiv f(v_1u_3)\pmod{2}$.

\ms\nt It is easy to check that $\{f(u_1u_2), f(u_2 u_3), f(u_3u_4)\}\ne \{2,4,6\}$. So we may assume that $f(v_1u_1)$ and $f(v_1u_3)$ are odd, and $f(v_1u_2)$ and $f(v_1u_4)$ are even. Under these conditions and from \eqref{eq-W4} we have
$9\le a\le 11$ and $8\le b\le 12$.
\begin{enumerate}[1.]
\item Suppose $a=9$. Then $f(v_1u_2)+f(v_1u_4)=10$ and hence $\{f(v_1u_2), f(v_1u_4)\}=\{4,6\}$. This implies that $f(u_1u_2)=2$ and $f(v_1u_1)=7$. If $f(v_1u_2)=4$ and $f(v_1u_4)=6$, then $f(u_2u_3)=f(u_3u_4)$ which is impossible. Thus $f(v_1u_2)=6$ and $f(v_1u_4)=4$. This implies that $9\le 2+6+f(u_2u_3)=b=4+f(u_3u_4)\le 9$. Hence $b=9=a$ which is a contradiction.

\item Suppose $a=10$. We have $\{f(v_1u_1), f(u_1u_2)\}=\{3,7\}$ and $\{f(v_1u_3),f(u_2u_3), f(u_3u_4)\}=\{1, 4, 5\}$. Since $f(v_1u_2)+f(v_1u_4)=8$,
      $\{f(v_1u_2),f(v_1u_4)\}=\{2,6\}$.
      Since $b\ge 8$, $f(v_1u_4)=6$. Hence $f(v_1u_2)=2$. Since $a\ne b$, $f(u_3u_4)=5$ and hence $f(u_2u_3)=4$. Now $f^+(u_2)\ne b=11$, which is a contradiction.
\item Suppose $a=11$. We have $f(v_1u_2)+f(v_1u_4)=6$. This implies that $\{f(v_1u_2),f(v_1u_4)\}=\{2,4\}$. Since 4 is occupied and $f(v_1u_1)+f(u_1u_2)=11$, $f(v_1u_1)=5$ and $f(u_1u_2)=6$. Also we have $\{f(v_1u_3), f(u_2u_3), f(u_3u_4)\}=\{1, 3, 7\}$. Since $b\ge 8$, $f(u_3u_4)=7$. Since $b\ne a$, $f(v_1u_4)=2$. Now $b=9$ and $f^+(u_2)\ge 10$ which yields
a contradiction. \end{enumerate}
As a conclusion, $\chi_{la}(W_4-e)\ge 4$. Note that from the discussion above, we have obtained a local antimagic labeling $g$ for $W_4-e$ with $c(g)=4$.
\end{proof}

\begin{theorem}\label{thm-Wm-e-even} Let $e$ be an edge of $W_m$. For even $m\ge 6$, $\chi_{la}(W_m - e)=3$.
\end{theorem}

\begin{proof}
\ms\nt Consider $m=6$. In Figure~\ref{fig:W6-e}, we have the local antimagic labelings $f$ with $c(f)=3$ for the two cases of $W_6-e$. \\[1mm]
\begin{figure}[H] 
  \centering
  \centerline{\epsfig{file=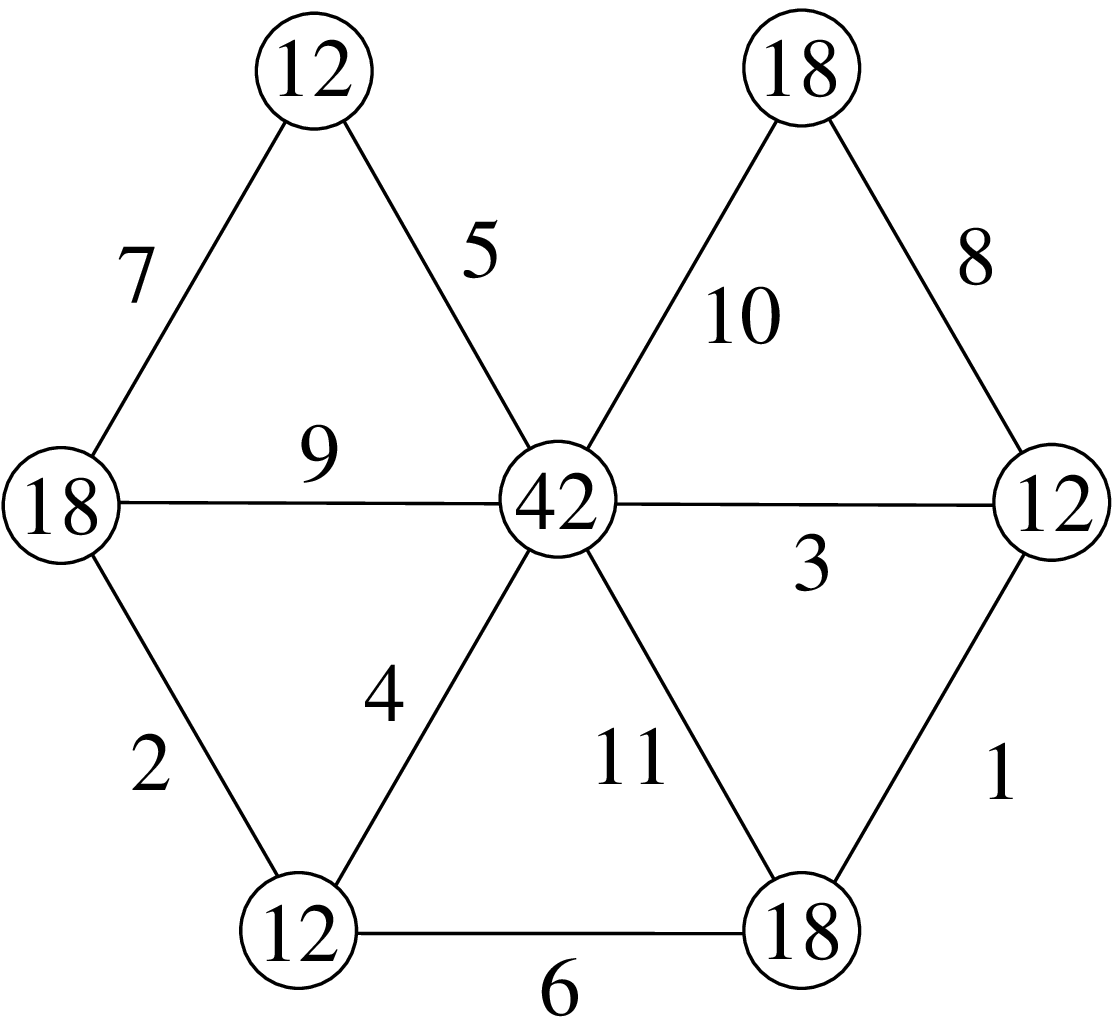, width=3cm}\qquad \epsfig{file=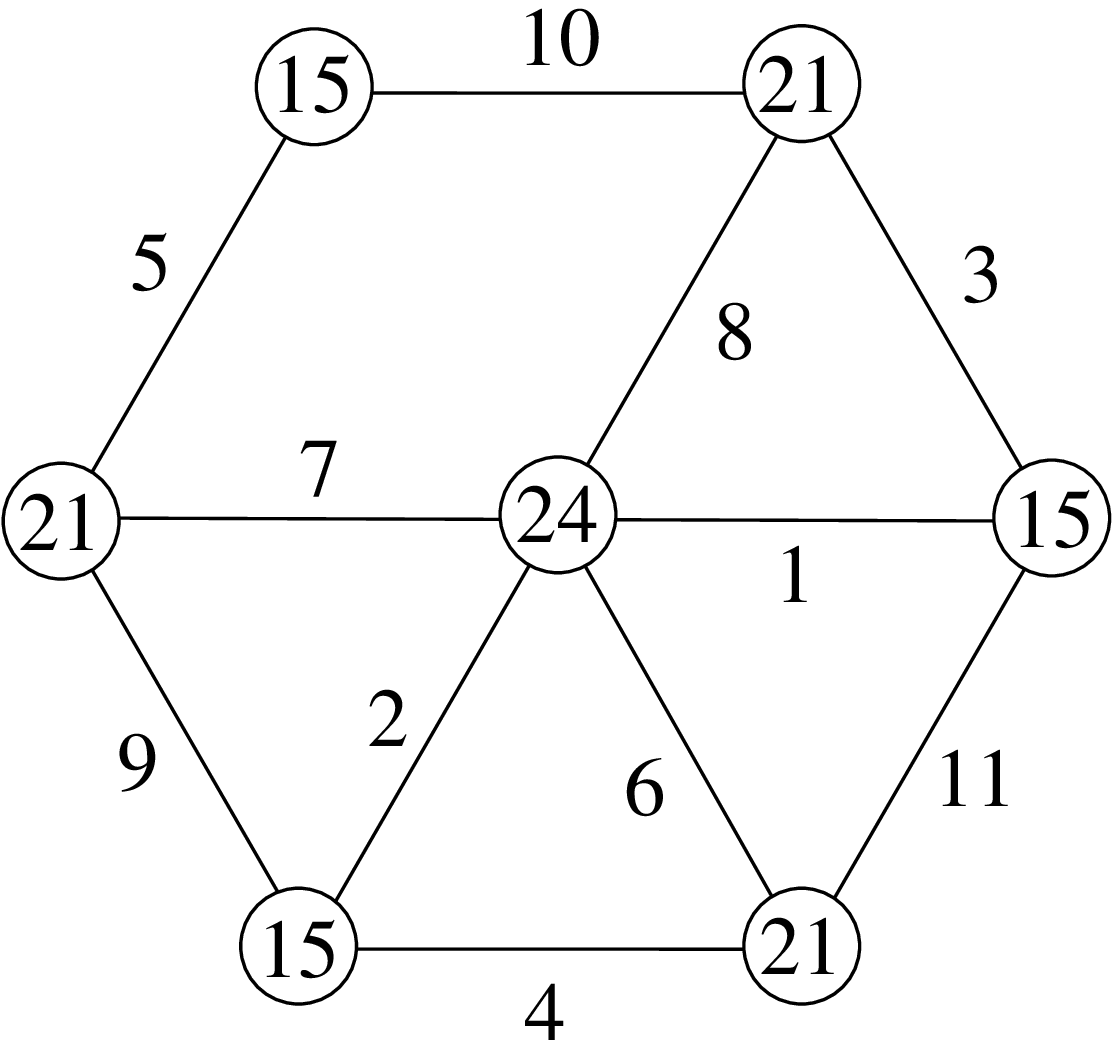, width=3cm}}
  \caption{$W_6 - e$ with $c(f)=3$}
  \label{fig:W6-e}
\end{figure}
\nt Thus, $\chi_{la}(W_6-e)=3$.

\ms\nt Consider $m\ge 8$. We have two cases.

\ms\nt Case (a). $e\in E(C_m)$. By~\cite[Theorem 6]{LNS} and~\cite[Theorem 2.14]{Arumugam} and the proofs, we have $\chi_{la}(W_m)=3$ such that the corresponding local antimagic labeling $f$ has $f(u_1u_2)=1$.
By symmetry we may let $e=u_1u_2$. By Lemma~\ref{lem-G-e}, we get $\chi_{la}(W_m - e)\le 3$. Since $\chi_{la}(W_m - e)\ge \chi(W_m-e)=3$, $\chi_{la}(W_m - e)=3$. 

\ms\nt Case (b). $e\not\in E(C_m)$. For $m=8$, the graph in Figure~\ref{fig:G8-1} shows that $W_8 - e$ admits a local antimagic labeling $g$ with $c(g)=3$. Thus, $\chi_{la}(W_8 - e)=3$.

\ms\nt Consider $m\ge 10$. By~\cite[Theorem 6]{LNS} and~\cite[Theorem 2.14]{Arumugam} and the proofs, we know that $W_m$ admits a local antimagic labeling $f$ with $f(v_1u_2)=2m$ if $m\equiv0\pmod{4}$, and $f(v_1u_4)=2m$ if $m\equiv 2\pmod{4}$. 
 By symmetry we may let $e=v_1u_2$ if $m\equiv 0\pmod{4}$, and $e=v_1u_4$ if $m\equiv2\pmod{4}$. It is straightforward to check the conditions of Lemma~\ref{lem-G-e}. By Lemma~\ref{lem-G-e}, we get $\chi_{la}(W_m - e)\le 3$. Since $\chi_{la}(W_m - e)\ge \chi(W_m-e)=3$, $\chi_{la}(W_m - e)=3$.
\end{proof}

\begin{figure}[H] 
\centering
\begin{subfigure}[$W_8 - v_1u_1$]{
\epsfig{file=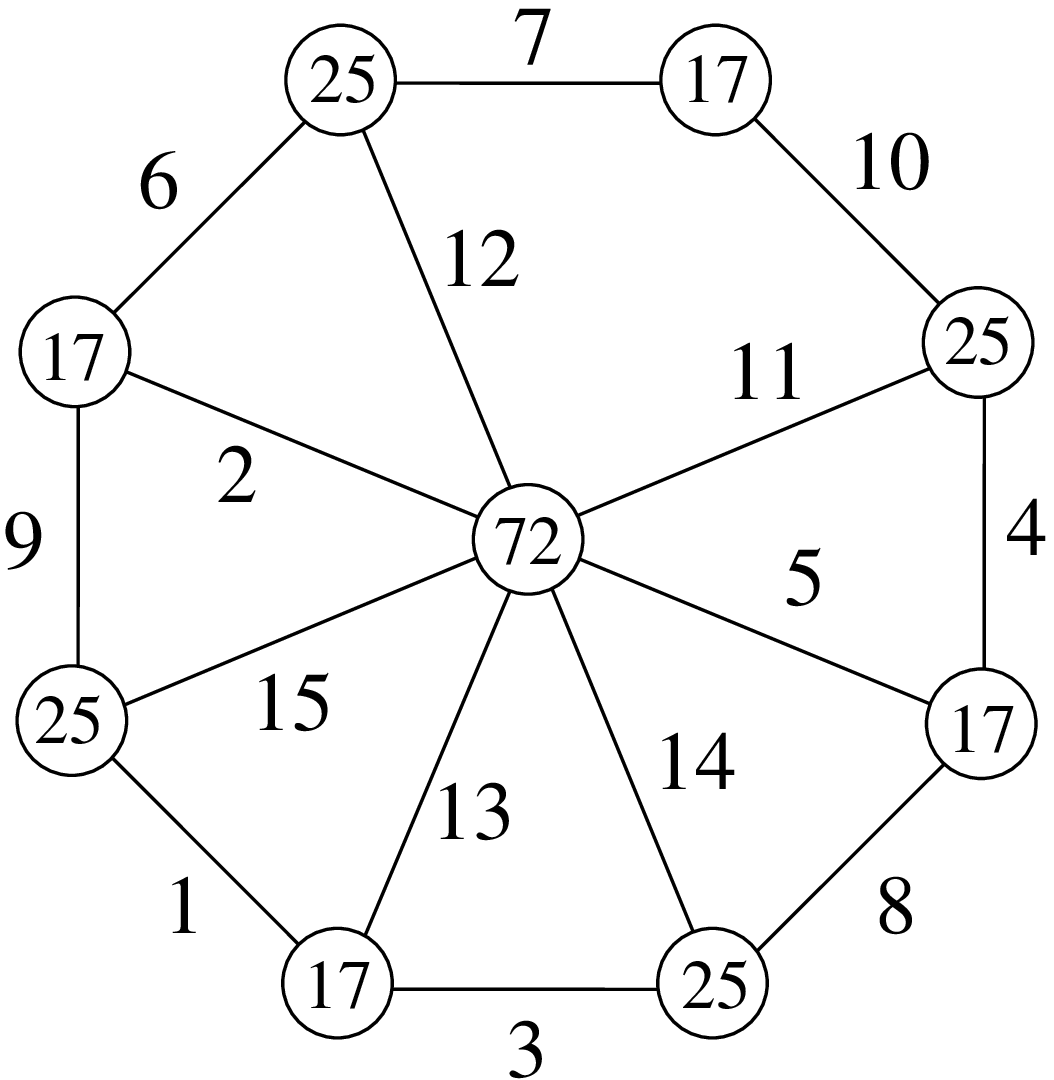, width=1.2in}\label{fig:G8-1}}
\end{subfigure}\qquad
 \begin{subfigure} [$W_5 - v_1u_5$.]{
\raisebox{3mm}[3mm][2mm]{\epsfig{file=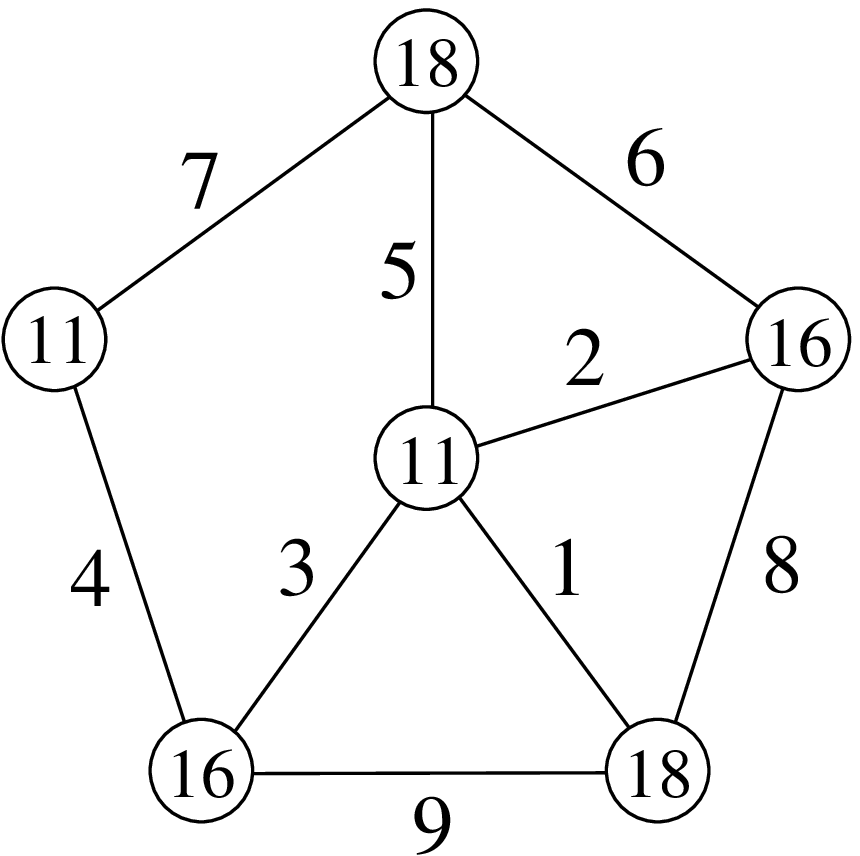, width=1in}\label{fig:W5-e2}}}
\end{subfigure} \qquad
\begin{subfigure} [$W_7 - v_1u_7$.]{
\raisebox{3mm}[3mm][2mm]{\epsfig{file=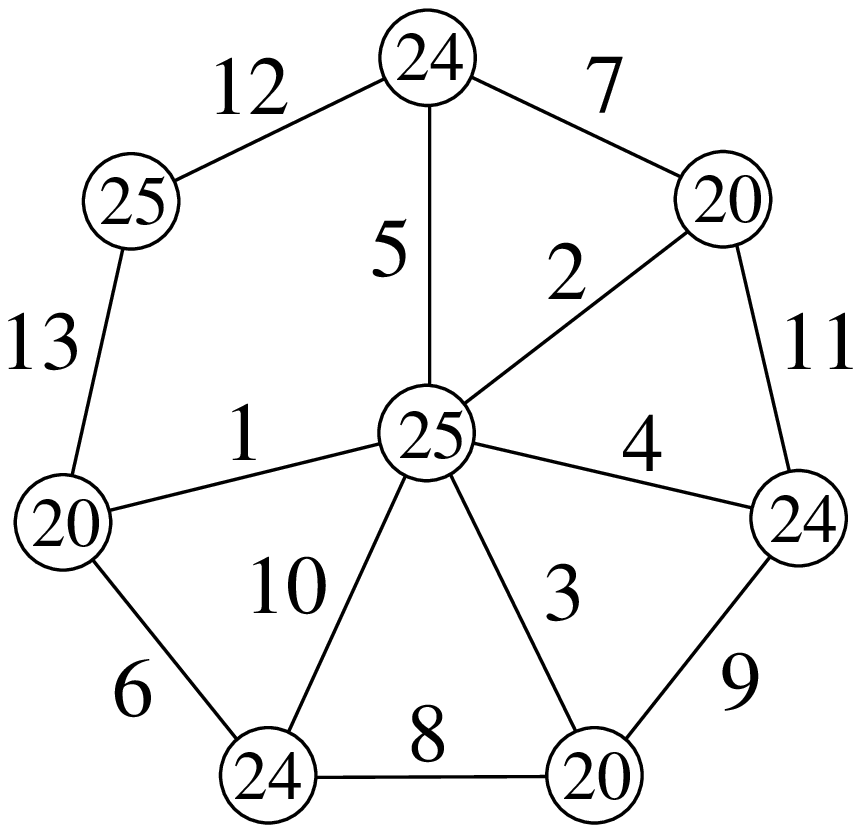, width=1.1in}\label{fig:W7-e2}}}
\end{subfigure}
\caption{Some wheels with a spoke deleted}
\end{figure}

\begin{theorem}\label{thm-oddwheel-e} Suppose $m\ge 3$ is odd. If $e\not\in E(C_m)$,  then $$\chi_{la}(W_m - e) = \begin{cases}3 & \mbox{ for }  m=3,5,7;\\ 4 & \mbox{ otherwise.}\end{cases}$$ If $e\in E(C_m)$, then $3\le \chi_{la}(W_m - e)\le 4$. \end{theorem}

\begin{proof} Suppose $e\not\in E(C_m)$. Note that $\chi_{la}(W_m - e)\ge\chi(W_m - e) = 3$. Suppose the equality holds. Let $m=2k+1$ and $f$ is a local antimagic labeling of $W_{2k+1}-e$ with $c(f)=3$. Without loss of generality, assume $e=v_1u_{2k+1}$. Thus, we must have $f^+(v_1) = f^+(u_{2k+1}) \not= f^+(u_1) = f^+(u_3) = \cdots = f^+(u_{2k-1}) \not= f^+(u_2) = f^+(u_4) = f^+(u_{2k})$. Thus, $k(2k+1)\le f^+(v_1)=f^+(u_{2k+1})\le 8k+1$ giving us $1\le k\le 3$. Thus, $\chi_{la}(W_{m} - e) \ge 4$ for $m\ge 9$.  For $m=3$, $W_3-e \cong K_{1,1,2}$. The labeling is obvious.  For $m=5$, the labeling in Figure~\ref{fig:W5-e2} shows that $\chi_{la}(W_5 - v_1u_5) = 3$. For $m=7$, the labeling in Figure~\ref{fig:W7-e2} shows that $\chi_{la}(W_7 - v_1u_7) = 3$.

\ms\nt Consider $m\ge 9$. By~\cite[Theorem 2.14]{Arumugam} and the proofs, we know that $W_m$ admits a local antimagic labeling $f$ with $c(f)=4$. Moreover, $f(v_1u_5)=2m$ if $m\equiv1\pmod{4}$, and $f(v_1u_2)=2m$ if $m\equiv3\pmod{4}$. It is straightforward to check the conditions of Lemmas~\ref{lem-nonreg} and~\ref{lem-G-e}. By Lemma~\ref{lem-nonreg}, we know $W_m$ admits a local antimagic labeling $g$ with $g(v_1u_5)=1$ if $m\equiv1\pmod{4}$, and $g(v_1u_2)=1$ if $m\equiv3\pmod{4}$. By Lemma~\ref{lem-G-e}, we get $\chi_{la}(W_m - e)= 4$.

\ms\nt Suppose $e\in E(C_m)$. By~\cite[Theorem 2.14]{Arumugam} and the proofs, together with Lemma~\ref{lem-G-e}, we know that $\chi_{la}(W_m-e)\le 4$. \end{proof}




\begin{theorem}\label{thm-CmVCnodd} For odd $m,n\ge 3$, $\chi_{la}(C_m\vee C_n) = 6$.
\end{theorem}

\begin{proof} Since $C_m \vee C_n$ and $C_n \vee C_m$ are isomorphic, we may assume that $n\le m$. Suppose $V(C_m \vee C_n) = V(C_m \vee O_n)$ and $E(C_m\vee C_n) = E(C_m \vee O_n) \cup \{e'_j = v_jv_{j+1}: 1\le j\le n\}$ as in Theorem~\ref{thm-CmOnodd}, where $v_{n+1}=v_1$. Let $f$ be the local antimagic labeling of $C_m \vee O_n$ defined in the proof of Theorem~\ref{thm-CmOnodd}. Define an edge labeling $g : E(C_m\vee C_n)\to [1,m+mn+n]$ such that $g(e) = f(e)$ for $e\in E(C_m\vee O_n)$ and $g(e'_j) = m + mn + f(e'_j)$. One may check that $g$ is a bijection. Moreover,
\begin{enumerate}[(i)]
  \item $g^+(u_1) = g_1=n(mn+1)/2 + mn + (m+3)/2$,
  \item $g^+(u_i) = g_2 = n(mn+1)/2 + mn + m+1$ for even $i$,
  \item $g^+(u_i) = g_3 = n(mn+1)/2 + mn + m+2$ for odd $i\ge 3$,
  \item $g^+(v_1) = g_4 = m(mn+1)/2 + m^2 + 2(m + mn)+ (n+3)/2$,
  \item $g^+(v_j) = g_5 = m(mn+1)/2 + m^2 + 2(m + mn)+ n + 1$ for even $j$, and
  \item $g^+(v_j) = g_6 = m(mn+1)/2 + m^2 + 2(m + mn)+ n + 2$ for odd $j\ge 3$.
\end{enumerate}

\nt Clearly $g_k < g_{k+1}$ for $1\le k\le 5$. Thus, $\chi_{la}(C_m\vee C_n) \le 6$. Since $\chi_{la}(C_m \vee C_n) \ge \chi(C_m \vee C_n) = \chi(C_m) + \chi(C_n) = 6$, we have  $\chi_{la}(C_m \vee C_n)=6$.
\end{proof}

\nt In~\cite{Haslegrave}, Haslegrave proved that every connected graph $G\ne K_2$ admits a local antimagic labeling which implies that $\chi_{la}(K_n)=n$ for all $n\ge 3$. We now consider the join graph $C_m\vee K_n$ with $V(C_m\vee K_n) = V(C_m\vee O_n)$ and $E(C_m\vee K_n) = E(C_m\vee O_n) \cup \{v_iv_j : 1\le i < j\le n\}$.  In~\cite{Arumugam}, the authors showed that $\chi_{la}(C_m \vee K_1) = 3$ for odd $m\ge 3$.

\begin{theorem}\label{thm-CmKnodd} For odd $m,n\ge 3$, $\chi_{la}(C_m \vee K_n) = n+3$.
\end{theorem}

\begin{proof} Let $f$ be the local antimagic labeling of $C_m \vee O_n$ defined in the proof of Theorem~\ref{thm-CmOnodd}. Let $h : E(K_n) \to [1, n(n-1)/2]$ be a local antimagic labeling of $K_n$. Note that $h^+(v_j)$ are distinct for $1\le j \le n$. Define an edge labeling $g : E(C_m\vee K_n) \to [1,mn+m+n(n-1)/2]$ such that $g(e) = f(e)$ for $e\in E(C_m\vee O_n)$ and $g(e) = h(e) + mn + m$ for $e\in E(K_n)$.
Note that $g^+(v_j) = f^+(v_j) + h^+(v_j)+(n-1)(mn+n)$. Since $f^+(v_j)$ are the same and $h^+(v_j)$ are distinct, $g^+(v_j)$ are distinct for $1\le j\le n$.

\ms\nt Moreover,
\begin{enumerate}[(i)]
  \item $g^+(u_1) = n(mn+1)/2 + mn + (m+3)/2$,
  \item $g^+(u_i) = n(mn+1)/2 + mn + m+1$ for even $i$,
  \item $g^+(u_i) = n(mn+1)/2 + mn + m+2$ for odd $i\ge 3$, and
  \item $g^+(v_j) = f^+(v_j) + h^+(v_j)+(n-1)(mn+n) \ge m(mn+1)/2 + m^2 + (n-1)(nm+m) + n(n-1)/2$.
\end{enumerate}

\nt It is easy to show that $g^+(v_j) > g^+(u_i)$ for all $1\le i\le m, 1\le j\le n$. Thus, $\chi_{la}(C_m \vee K_n) \le n+3$.  Since $\chi_{la}(C_m \vee K_n) \ge \chi(C_m \vee K_n) = n+3$, the theorem holds.
\end{proof}

\begin{theorem} For $m\ge 2, n\ge 1$, $\chi_{la}(C_{2m}\vee K_{2n}) = 2n+2$.
\end{theorem}

\begin{proof} Let $f$ be the local antimagic labeling of $C_{2m}\vee O_{2n}$ defined in the proof of Theorem~\ref{thm-CmOneven}.

\ms\nt Suppose $n=1$. Define an edge labeling $g : E(C_{2m}\vee K_2) \to [1,6m+1]$ such that $g(e) = f(e)$ for $e\in E(C_{2m}\vee O_2)$ and $g(v_1v_2) = 6m+1$. We now swap the labels of $g(u_1v_1)=2m+1$ and $g(u_1v_2)=6m-1$ to get $g^+(u_{2i-1}) = 10m+1$ and $g^+(u_{2i}) = 10m+3$ for $1\le i\le m$ and $g^+(v_1) = 8m^2 + 11m - 1$ and $g^+(v_2) = 8m^2 + 3m + 3$. Thus, $\chi_{la}(C_{2m} \vee K_{2}) \le 4$.

\ms\nt Now, consider $n\ge 2$. Let $h : E(K_{2n}) \to [1, n(2n-1)]$ be a local antimagic labeling of $K_{2n}$. Note that $h^+(v_j)$ are distinct for $1\le j \le 2n$. Define an edge labeling $g : E(C_{2m}\vee K_{2n}) \to [1,4mn+2m+n(2n-1)]$ such that $g(e) = f(e)$ for $e\in E(C_{2m}\vee O_{2n})$ and $g(e) = h(e) + 4mn + 2m$ for $e\in E(K_{2n})$.

\ms\nt By the same argument in the proof of Theorem~\ref{thm-CmKnodd}, we obtain that $g^+(v_j)$ are distinct for $1\le j\le 2n$.

\ms\nt From Theorem~\ref{thm-CmOneven} we have
$g^+(u_{2i}) = 4mn^2 - 4mn + 2n + 10m - 1 < g^+(u_{2i-1}) =4mn^2+ 12mn - 6m + 3$ for $1 \le i\le m$.
Moveover, $g^+(v_j) = f^+(v_j) + h^+(v_j)+(2n-1)(4mn+2m) \ge 4m^2n + 4m^2 + m + (2n-1)(4mn + 2m) + n(2n-1)$ for each $j$. Clearly $g^+(v_j)> g^+(u_{2i-1})$ for $1\le i\le m$ and $1\le j\le 2n$.

\ms\nt Thus, $\chi_{la}(C_{2m} \vee K_{2n}) \le 2n+2$.  Since $\chi_{la}(C_{2m}\vee K_{2n}) \ge \chi(C_{2m} \vee K_{2n}) = 2n+2$, the theorem holds.
\end{proof}


\begin{conjecture} For $n\ge 2$, $\chi_{la}(G \vee K_n) \ge \chi_{la}(G) + n$ if and only if $\chi_{la}(G) = \chi(G)$. \end{conjecture}


\ms\nt For $n\ge 2$, let $M_{2n}$ be the Mobi\"{u}s ladder obtained from $C_{2n} = u_1u_2\cdots u_nv_1v_2\cdots v_nu_1$ by adding the edges $u_iv_i, 1\le i\le n$.

\begin{theorem}\label{thm-M2nodd} For odd $n\ge 3$, $\chi_{la}(M_{2n})=3$.  \end{theorem}

\begin{proof} Note that $M_{2n}$ has size $3$n, and is bipartite with parts of the same size. Thus, by Lemma~\ref{lem-2part}, $\chi_{la}(M_{2n})\ge 3$.

\ms\nt Suppose $n=3$, we get a local antimagic labeling by assigning the edges $u_1u_2$, $u_2u_3$, $u_3v_1$, $v_1v_2$, $v_2v_3$, $v_3u_1$, $u_1v_1$, $u_2v_2$, $u_3v_3$ by $1,5,4,8,6,7,3,9,2$, respectively. Clearly, the induced vertex coloring has three distinct colors, namely 11, 15, 23.

\ms\nt Suppose $n\ge 5$. Define a bijection $f : E(M_{2n})\to [1,3n]$ such that $f(u_1v_n)=\frac{3(n+1)}{2}$, $f(u_nv_1) = n$, $f(v_1v_2)=n+1$ and that
\begin{enumerate}[(i)]
  \item $f(u_iu_{i+1})=i$  for odd $i\in [ 1, n-2]$,
  \item $f(u_iu_{i+1})=\frac{3n+3-i}{2}$ for even $i \in[2,n-1]$,
  \item $f(v_iv_{i+1})=i$ for even $i \in [2, n-1]$,
  \item $f(v_iv_{i+1})=2n-\frac{i-3}{2}$ for odd $i\in [3, n-2]$,
  \item $f(u_iv_i) = \frac{5n+2-i}{2}$ for odd $i \in [1, n]$,
  \item $f(u_iv_i) = 3n+1 - \frac{i}{2}$ for even $i\in [2, n-1]$.
\end{enumerate}

\nt One can verify that $f^+(u_i) = f^+(v_j) = \frac{9n+3}{2}$ for even $i\in[2,n-1]$ and odd $j\in [1, n]$; $f^+(u_i)= f^+(v_2) = 4n+3$ for odd $i\in [1,n]$ and $f^+(v_j) = 5n+3$ for even $j\in [4,n-1]$. Therefore, $\chi_{la}(M_{2n}) \le 3$. Hence, the theorem holds.
\end{proof}

\begin{corollary} For odd $n\ge 3$, $\chi_{la}(M_{2n} - e) = 3$. \end{corollary}

\begin{proof} By Lemma~\ref{lem-2part}, we know that $\chi_{la}(M_{2n} - e)\ge 3$. Note that there are two possible graphs obtained by deleting  an edge from $M_{2n}$ (if $n > 3$), but using Lemma~\ref{lem-reg} with reference to the smallest label deals with one, 
and the largest label deals with the other. Therefore, we have $\chi_{la}(M_{2n} - e)\le 3$. Thus,  $\chi_{la}(M_{2n} - e) = 3$.
\end{proof}

\nt Note that $M_4 = K_4$ with $\chi_{la}(M_4)=4$.

\begin{conjecture} For even $n\ge 4$, $\chi_{la}(M_{2n}) = 4$. \end{conjecture}

\begin{theorem}\label{thm-M6VO2n} For $n\ge 1$, $\chi_{la}(M_6 \vee O_{2n}) = 3$. \end{theorem}

\begin{proof} Let $V(M_6\vee O_{2n})=\{u_i : 1\le i\le 6\}\cup \{v_j : 1\le j\le 2n\}$ and $E(M_6\vee O_{2n}) = \{u_iu_{i+1} : 1\le i\le 5\} \cup \{u_1u_6, u_1u_4, u_2u_5, u_3u_6\} \cup \{u_iv_j : 1\le i\le 6, 1 \le j\le 2n\}$. Define a bijection $g : E(M_6 \vee O_{2n}) \to [1, 12n + 9]$ such that $g(u_1u_2) = 1$, $g(u_2u_3) = 3$, $g(u_3u_4) = 4$, $g(u_4u_5) = 2$, $g(u_5u_6) = 8$, $g(u_1u_6) = 5$, $g(u_1u_4) = 9$, $g(u_2u_5) = 7$, $g(u_3u_6) = 6$ and $g(u_iv_j) = f(u_iv_j) + 3$ for $1\le i\le 6, 1\le j\le 2n$, where $f$ is the function as defined in the proof of Theorem~\ref{thm-CmOneven} by taking $m=3$.

\ms\nt One can easily check that $g^+(u_1) = 15 + \sum_{j=1}^{2n} f(u_1v_j) + 3(2n) = 12n^2-4n+37$. Similarly, we get $g^+(u_3) = g^+(u_5) = g^+(u_1)$. Furthermore, for $i=2,4,6$, we also have $g^+(u_i) = 12n^2 + 42n - 7$, whereas $g^+(v_j) = 36n+57$ for $1\le j\le 2n$. Clearly, $g$ is a local antimagic labeling with $c(g) = 3$. Therefore, $\chi_{la}(M_6\vee O_{2n}) \le 3$. Since $M_6$ is bipartite, we have $\chi_{la}(M_6 \vee O_{2n}) \ge \chi(M_6\vee O_{2n}) = \chi(M_6) + \chi(O_{2n}) = 3$. Thus, $\chi_{la}(M_6 \vee O_{2n}) = 3$.
\end{proof}

\begin{corollary} For $n\ge 1$, $\chi_{la}((M_6\vee O_{2n})-e) = 3$.
\end{corollary}

\begin{proof} Let $G=(M_6\vee O_{2n})-e$. We note that $\chi_{la}(G)\ge \chi(G)=3$. Since $M_6$ is edge-transitive, we only need to consider (i) $e\not\in E(M_6)$, and (ii) $e\in E(M_6)$.

\ms\nt In (i), it is straightforward to check the conditions of Lemma~\ref{lem-nonreg}. By Lemma~\ref{lem-nonreg}, we know $M_6\vee O_{2n}$ admits a local antimagic labeling $h=12n+10 - g$ with $c(h) = c(g) = 3$, where $g$ is as defined in the proof of Theorem~\ref{thm-M6VO2n}. Now, $$h^+(u_i) = \begin{cases}
12n^2+60n-7 & \mbox{ if } i = 1,3,5,\\
12n^2+14n+37 & \mbox{ if } i = 2,4,6,
\end{cases}$$
$h^+(v_j)=36n+3$ for $1\le j\le 2n$, and $h(uv)=1$ for an edge $uv\not\in E(M_6)$. It is straightforward to check the condition of Lemma~\ref{lem-G-e}. By Lemma~\ref{lem-G-e}, we have $\chi_{la}(G) = 3$.

\ms\nt In (ii), it is straightforward to check the condition of Lemma~\ref{lem-G-e}. By Lemma~\ref{lem-G-e}, we have $\chi_{la}(G) = 3$.
\end{proof}


\ms\nt For $m\ge 3$, $n\ge 1$, let $G(m,n)$ be the graph obtained from $C_m \vee O_n$ by deleting the edges $u_mv_j$, $1\le j\le n$. Note that $G(m,1)$ is the graph $W_m$ with a spoke deleted. By Theorems~\ref{thm-W4-e} and \ref{thm-Wm-e-even}, we have $\chi_{la}(G(2m,1))=3$  for $m\ge 2$. Moreover, by Theorem~\ref{thm-oddwheel-e}, we have determined the value of $\chi_{la}(G(2m+1,1))$ for $m\ge 1$.


\begin{theorem}\label{thm-G4n}  For $n\ge 1$, $\chi_{la}(G(4, n)) = 3$. \end{theorem}

\begin{proof} When $n=1$, we have proved the result in Theorem~\ref{thm-W4-e}. So we may assume that $n\ge 2$. Since $\chi(G(4, n))\ge 3$, it suffices to provide a local antimagic labeling $f$ for $G(4, n)$ with $c(f)=3$.

\ms\nt For $n=4k-1$, $k\ge 1$, the labeling matrix of $G(4,3)$ under $f$ is given below.
\[\begin{array}{c|*{4}{c}|*{3}{c}|c}
& u_1 & u_2 & u_3 & u_4 & v_1 & v_2 & v_3 & f^+(u_i)\\\hline
u_1 & * & 8 & * & 9 & 5 & 1 & 13 & 36\\
u_2 & 8 & * & 7 & * & 3 & 12 & 4 & 34\\
u_3 & * & 7 & * & 10 & 11 & 6 & 2 & 36\\
u_4 & 9 & * & 10 & * & * & * & * & 19\\\hline
f^+(v_j) & * & * & * & * & 19 & 19 & 19
\end{array}\]

\ms\nt The following tables are the first 4 rows of the labeling matrix of $G(4,4k-1)$ under $f$, where $k\ge 3$.

{\fontsize{9}{11}\selectfont
\[\begin{array}{c|*{4}{@{\;\;}c}|*{4}{@{\;\;}c}|*{4}{@{\;\;}c}|}
& u_1 & u_2 & u_3 & u_4 & v_1 & v_2 & \cdots & v_{k} & v_{k+1} & v_{k+2} & \cdots & v_{2k} \\\hline
u_1 & * & 10k+1 & * & 6k & 8k & 8k-1 & \cdots & 7k+1 & 9k & 9k - 1 & \cdots & 8k+1\\\hline
u_2 & 10k+1 & * &  4k & * & 1 &  3 & \cdots & 2k-1 & 2k+1 & 2k+3 & \cdots & 4k-1 \\\hline
u_3 & * & 4k & * & 12k+1 & 10k & 10k-1 & \cdots & 9k + 1 & 7k & 7k-1 & \cdots & 6k+1 \\\hline
u_4 & 6k & * & 12k+1 & * & * & * & \cdots  & * & * & * & \cdots & *  \\\hline
\end{array}\]
\[\begin{array}{c|*{4}{@{\;\;}c}|*{4}{@{\;\;}c}|*{3}{@{\;\;}c}|c}
 & v_{2k+1} & v_{2k+2} & \cdots & v_{3k-2} & v_{3k-1} & v_{3k} & \cdots & v_{4k-4} & v_{4k-3} & v_{4k-2} & v_{4k-1} & f^+(u_i)\\\hline
u_1 & 12k & 12k-1 & \cdots & 11k + 3 & 5k + 1 & 5k & \cdots & 4k+4 & 4k-6 & 4k+2 & 4k + 1 & 32k^2 + k - 10 \\\hline
u_2 & 2 & 4 & \cdots & 2k-4 & 2k-2 & 2k & \cdots & 4k-8 & 10k+4 & 10k+3 & 10k+2 & 8k^2 + 16k + 21 \\\hline
u_3 & 6k-1 & 6k-2 & \cdots & 5k+2 & 11k+2 & 11k + 1 & \cdots & 10k+5 &  4k+3 & 4k-4 & 4k-2 & 32k^2 + k - 10 \\\hline
u_4 & * & * & \cdots & * & * & * & \cdots &  * & * & * & * & 18k+1 \\\hline
\end{array}\]
}

\ms\nt It is easy to check that $f^+(u_4)=f^+(v_j)=18k+1$, i.e., the $v_j$-column sum, for $1\le j\le 4k-1$. This labeling can be applied to $k=2$ (the block-columns for $v_{2k+1}$ to $v_{4k-4}$ do not appear). The following shows the assignment for $G(4,7)$:
{\fontsize{9}{11}\selectfont
\[\begin{array}{c|*{4}{c}|*{2}{c}|*{2}{c}|*{3}{c}|c}
& u_1 & u_2 & u_3 & u_4 & v_1 & v_2 & v_3 & v_{4} & v_{5} & v_{6} & v_7 & f^+(u_i) \\\hline
u_1 & * & 21 & * & 12 & 16 & 15 & 18 & 17 & 2 & 10 & 9 & 120\\\hline
u_2 & 21 & * & 8 & * & 1 &  3 & 5 & 7 & 24 & 23 & 22 & 114 \\\hline
u_3 & * & 8 & * & 25 & 20 & 19 & 14 & 13 & 11 & 4 & 6 & 120 \\\hline
u_4 & 12 & * & 25 & * & * & * & *  & * & * & * & * & 37  \\\hline
f^+(v_j) &  * & * & * & * & 37 & 37 & 37 &  37 & 37 & 37 & 37 &
\end{array}\]
}

\ms\nt For $n=4k+1$, $k\ge 1$, the labeling matrix for $G(4,5)$ is given below.
\[\begin{array}{c|*{4}{c}|*{5}{c}|c}
& u_1 & u_2 & u_3 & u_4 & v_1 & v_2 & v_3 & v_4 & v_5 & f^+(u_i)\\\hline
u_1 & * & 4 & * & 16 & 10 & 9 & 8 & 11 & 13 & 71\\
u_2 & 4 & * & 6 & * & 1 & 3 & 17 & 12 & 15 & 58\\
u_3 & * & 6 & * & 14 & 19 & 18 & 5 & 7 & 2 & 71\\
u_4 & 16 & * & 14 & * & * & * & * & * & * & 30\\\hline
f^+(v_j) & * & * & * & * & 30 & 30 & 30 & 30 & 30
\end{array}\]

\ms\nt Similarly, we show the first 4 rows of the labeling matrix of $G(4,4k+1)$ under $f$, where $k\ge 3$.

{\fontsize{9}{11}\selectfont
\[\begin{array}{c|*{4}{@{\;\;}c}|*{4}{@{\;\;}c}|*{4}{@{\;\;}c}|}
& u_1 & u_2 & u_3 & u_4 & v_1 & v_2 & \cdots & v_{k-2} & v_{k-1} & v_{k} & \cdots & v_{2k-4}  \\\hline
u_1 & * & 10k+6 & * & 12k+7 & 8k+4 & 8k+3 & \cdots & 7k+7 & 9k+7 & 9k+6 & \cdots & 8k+10  \\\hline
u_2 & 10k+6 & * & 4k+2 & * & 1 &  3 & \cdots & 2k-5 & 2k-3 & 2k-1 & \cdots & 4k-9 \\\hline
u_3 & * & 4k+2 & * & 6k+3 & 10k+5 & 10k+4 & \cdots & 9k+8 & 7k+6 & 7k+5 & \cdots & 6k+9\\\hline
u_4 & 12k+7 & * & 6k+3 & * & * & * & \cdots  & *  & * & * & \cdots & * \\\hline
\end{array}\]
\[\begin{array}{c|*{4}{@{\;\;}c}|c|}
&  v_{2k-3} & v_{2k-2} & v_{2k-1} & v_{2k} & v_{2k+1} \\\hline
u_1 & 6k+8 & 6k+7 & 6k+6 & 6k+5 & 4k+1 \\\hline
u_2 &  4k-7 & 4k-5 & 4k-3 & 4k-1 & 6k+4 \\\hline
u_3 & 8k+9 & 8k+8 & 8k+7 & 8k+6 & 8k+5 \\\hline
u_4 &  *  & * & * & * & * \\\hline
\end{array}\]
\[\begin{array}{c|*{4}{@{\;\;}c}|*{4}{@{\;\;}c}|c}
 & v_{2k+2} & v_{2k+3} & \cdots & v_{3k+1} & v_{3k+2} & v_{3k+3} & \cdots &  v_{4k+1} & f^+(u_i)\\\hline
u_1 & 12k+6 & 12k+5 & \cdots & 11k+7 & 5k+2 & 5k+1 & \cdots & 4k+3 & 32k^2+41k+12 \\\hline
u_2 & 2 & 4 & \cdots & 2k & 2k+2 & 2k+4 & \cdots & 4k & 8k^2+22k+12 \\\hline
u_3 & 6k+2 & 6k+1 & \cdots & 5k+3 & 11k+6 & 11k+5 & \cdots & 10k+7 & 32k^2+41k+12 \\\hline
u_4 & * & * & \cdots & * &* & * & \cdots &  *  & 18k+10 \\\hline
\end{array}\]
}

\ms\nt It is easy to check that $f^+(u_4)=f^+(v_j)=18k+10$, for $1\le j\le 4k+1$. This labeling can be applied to $k=2$ (the block-columns for $v_1$ to $v_{2k-4}$ do not appear). The following shows the assignment for $G(4,9)$:
\[\begin{array}{c|*{4}{@{\;\;}c}|*{4}{@{\;\;}c}|c|*{2}{c}|*{2}{c}|c}
& u_1 & u_2 & u_3 & u_4 & v_1 & v_2 & v_3 & v_{4} & v_{5} & v_{6} & v_7 & v_8 & v_9 & f^+(u_i) \\\hline
u_1 & * & 26 & * & 31 & 20 & 19 & 18 & 17 & 9 & 30 & 29 & 12 & 11 & 222\\\hline
u_2 & 26 & * & 10 & * & 1 &  3 & 5 & 7 & 16 & 2 & 4 & 6 & 8 & 88 \\\hline
u_3 & * & 10 & * & 15 & 25 & 24 & 23 & 22 & 21 & 14 & 13 & 28 & 27 & 222 \\\hline
u_4 & 10 & * & 15 & * & * & * & *  & * & * & * & * & * & * & 46  \\\hline
f^+(v_j) &  * & * & * & * & 46 & 46 & 46 & 46 & 46 & 46 & 46 & 46 & 46
\end{array}\]

\ms\nt For $n=4k+2$, the following tables are the first 4 rows of the labeling matrix of $G(4,4k+2)$ under $f$, where $k\ge 1$.
{\fontsize{9}{11}\selectfont
\[\begin{array}{c|*{4}{@{\;\;}c}|*{4}{@{\;\;}c}|*{4}{@{\;\;}c}|}
& u_1 & u_2 & u_3 & u_4 & v_1 & v_2 & \cdots & v_k & v_{k+1} & v_{k+2} & \cdots & v_{2k} \\\hline
u_1 & * & 8k+6 & * & 12k+9 & 10k+7 & 10k+6 & \cdots & 9k+8 & 7k+5 & 7k+4 & \cdots & 6k+6  \\\hline
u_2 & 8k+6 & * & 12k+10 & * & 1 &  3 & \cdots & 2k-1 & 2k+1 & 2k+3 & \cdots & 4k-1 \\\hline
u_3 & * & 12k+10 & * & 6k+4 & 8k+5 & 8k+4 & \cdots & 7k+6 & 9k+7 & 9k+6 & \cdots & 8k+8 \\\hline
u_4 & 12k+9 & * & 6k+4 & * & * & * & \cdots  & * & * & * & \cdots & *\\\hline
\end{array}\]
\[\begin{array}{c|c|*{4}{@{\;\;}c}|*{4}{@{\;\;}c}|c|c}
& v_{2k+1} & v_{2k+2} & v_{2k+3} & \cdots & v_{3k+1} & v_{3k+2} & v_{3k+3} & \cdots & v_{4k+1} & v_{4k+2} & f^+(u_i)\\\hline
u_1 & 6k+5 & 12k+8 & 12k+7 & \cdots & 11k+9 & 5k+3 & 5k+2 & \cdots & 4k+4 & 4k+3 & 32k^2+55k+23 \\\hline
u_2 & 4k+1 & 2 & 4 & \cdots & 2k & 2k+2 & 2k+4 & \cdots & 4k & 10k+8 & 8k^2 + 36k + 25 \\\hline
u_3 & 8k+7 & 6k+3 & 6k+2 & \cdots & 5k+4 & 11k+8 & 11k+7 & \cdots & 10k+9 &  4k+2 & 32k^2+55k+23 \\\hline
u_4  & * & * & * & \cdots & * & * & * & \cdots &  * & * & 18k+13 \\\hline
\end{array}\]}

\ms\nt It is easy to check that $f^+(u_4)=f^+(v_j)=18k+13$, for $1\le j\le 4k+2$. This labeling can be applied to $k=0$. The following shows the assignment  for $G(4,2)$:
\[\begin{array}{c|*{4}{@{\;\;}c}|cc|c}
& u_1 & u_2 & u_3 & u_4 & v_1 & v_2 & f^+(u_i) \\\hline
u_1 & * & 6 & * & 9 & 5 & 3 & 23 \\\hline
u_2 & 6 & * & 10 & * & 1 & 8 & 25 \\\hline
u_3 & * & 10 & * & 4 & 7 & 2 & 23 \\\hline
u_4 & 10 & * & 4 & * & * & * & 13 \\\hline
f^+(v_j) &  * & * & * & * & 13 & 13
\end{array}\]

\ms\nt For $n=4k$, the following tables are the first 4 rows of the labeling matrix of $G(4,4k)$ under $f$, where $k\ge 2$.
{\fontsize{9}{11}\selectfont
\[\begin{array}{c|*{4}{@{\;\;}c}|*{4}{@{\;\;}c}|*{4}{@{\;\;}c}|c|}
& u_1 & u_2 & u_3 & u_4 & v_1 & v_2 & \cdots & v_{k-1} & v_{k} & v_{k+1} & \cdots & v_{2k-2} & v_{2k-1} \\\hline
u_1 & * & 10k+3 & * & 12k+4 & 10k+2 & 10k+1 & \cdots & 9k+4 & 7k+3 & 7k+2 & \cdots & 6k+5 & 6k+4\\\hline
u_2 & 10k+3 & * & 6k+2 & * & 1 &  3 & \cdots & 2k-3 & 2k-1 & 2k+1 & \cdots & 4k-5 & 4k-3\\\hline
u_3 & * & 6k+2 & * & 6k+1 & 8k+2 & 8k+1 & \cdots & 7k+4 & 9k+3 & 9k+2 & \cdots & 8k+5 & 8k+4 \\\hline
u_4 & 12k+4 & * & 6k+1 & * & * & * & \cdots  & * & * & * & \cdots & * & *\\\hline
\end{array}\]
\[\begin{array}{c|c|*{4}{@{\;\,}c}|*{4}{@{\;\,}c}|c|c|c}
& v_{2k} & v_{2k+1} & v_{2k+2} & \cdots & v_{3k-1} & v_{3k} & v_{3k+1} & \cdots & v_{4k-2} & v_{4k-1} & v_{4k} & f^+(u_i)\\\hline
u_1 & 6k+3 & 12k+3 & 12k+2 & \cdots & 11k+5 & 5k+1 & 5k & \cdots & 4k+3 & 4k+2 & 4k & 32k^2+23k+3 \\\hline
u_2 & 4k-1 & 2 & 4 & \cdots & 2k-2 & 2k & 2k+2 & \cdots & 4k-4 & 4k-2 & 10k+4 & 8k^2 + 24k + 9 \\\hline
u_3 & 8k+3 & 6k & 6k-1 & \cdots & 5k+2 & 11k+4 & 11k+3 & \cdots & 10k+6 & 10k+5 & 4k+1 & 32k^2+23k+3 \\\hline
u_4 & * & * & * & \cdots & * & * & * & \cdots &  * & * & * & 18k+5 \\\hline
\end{array}\]}

\ms\nt It is easy to check that $f^+(u_4)=f^+(v_j)=18k+5$, for $1\le j\le 4k$. Again, this labeling can be applied to $k=1$. The following shows the assignment  for $G(4,4)$:
\[\begin{array}{c|*{4}{@{\;\;}c}|*{4}{@{\;\;}c}|c}
& u_1 & u_2 & u_3 & u_4 & v_1 & v_2 & v_3 & v_4 & f^+(u_i) \\\hline
u_1 & * & 13 & * & 16 & 10 & 9 & 6 & 4 & 58 \\\hline
u_2 & 13 & * & 8 & * & 1 & 3 & 2 & 14 & 41 \\\hline
u_3 & * & 8 & * & 7 & 12 & 11 & 15 & 5 & 58 \\\hline
u_4 & 16 & * & 7 & * & * & * & * & * & 23 \\\hline
f^+(v_j) &  * & * & * & * & 23 & 23 & 23 & 23  \\
\end{array}\]

\ms\nt Since $f^+(u_1)=f^+(u_3)\ne f^+(u_2)\ne f^+(u_4)$, we have $c(f)=3$. The proof is complete.\end{proof}

\nt Note that $P_3\vee O_{n+1}$ can be obtained from $G(4,n)$ by adding the edge $u_2u_4$. By Lemma~\ref{lem-G+e}, the following is obtained.

\begin{corollary} If $G\equiv P_3\vee O_{n+1}$, then $\chi_{la}(G)=3$. \end{corollary}

\begin{problem} Determine $\chi_{la}(P_m\vee O_n)$ for $m\ge 4, n\ge 2$. \end{problem}

\begin{theorem}\label{thm-Gmevennodd} For (i) $m\ge 3,$ $n\ge 4$, (ii) $m\ge 21,$ $n=3$, and (iii) $m\ge 4,$ $n=2$,  $\chi_{la}(G(2m, 2n-1)) = 4$. \end{theorem}

\begin{proof} Note that $\chi_{la}(G(2m, 2n-1)) \ge \chi(G(2m, 2n-1)) = 3$. Suppose $f$ is a local antimagic labeling of $G(2m, 2n-1)$ with $c(f) = 3$. We may have (I) $a=f^+(u_{2i-1}), 1\le i\le m$; $b = f^+(v_j) = f^+(u_{2m}), 1\le j\le 2n-1$; $c = f^+(u_{2i}), 1 \le i < m$ ; or (II) $a = f^+(u_{2i-1}), 1\le i\le m$; $b = f^+(v_j), 1\le j\le 2n-1$; $c = f^+(u_{2i}), 1\le i \le m$. Here $a, b, c$ are distinct. Now, every $v_j$ is adjacent to $2m-1$ vertices of $C_{2m}$.

\ms\nt For (I), $\sum_{j=1}^{2n-1} f^+(v_j)\ge 1 + 2 + \cdots + (2n-1)(2m-1) = (2n-1)(2m-1)(2mn-m-n+1)$. So, \begin{equation}\label{eq-Geo1}
(2m-1)(2mn-m-n+1)\le b=f^+(u_{2m})\le 8mn-4n+1\end{equation} giving $n\le \frac{(2m-1)(m-1)+1}{(2m-1)(2m-5)}$.
By simple calculus, we have $n\le \frac{11}{5}$. When $n=2$, we get $m=3$. This is not a case.

\ms\nt For (II), there are exactly $(2n-1)(m-1)+2m-2=2mn+m-2n-1$ edges incident to the vertices $u_{2i}$ for $1\le i\le m-1$. Each label of these edges contributes to the sum $\sum_{i=1}^{m-1} f^+(u_{2i})$ exactly once. Thus, $(m-1)c\ge \frac{1}{2}(2mn+m-2n-1)(2mn+m-2n)$. Therefore, we will get \begin{equation}\label{eq-Geo2}(2n+1)(2mn +m- 2n)\le 2c=2f^+(u_{2m})\le 16mn-8n+2.\end{equation}
However, $(2n+1)(2mn +m- 2n)\ge 11(2mn+m-2n)\ge 16mn +18n+11m-22n=16mn-4n+11m$, contradicting \eqref{eq-Geo2}, if $n\ge 5$. When $n=4$, we get $m=2$, contradicting $m\ge 3$. When $n=3$, we get $2\le m\le 20$, contradicting $m\ge 21$. So, $\chi_{la}(G(2m, 2n-1)) \ge 4$ under each of the given condition.

\ms\nt Define $f: E(G(2m, 2n-1)) \to [1, 4mn-2n+1]$ such that $f(u_{2m})= (2m-1)(2n-1) + 1$, $f(u_{2i}) = (2m-1)(2n-1) + i + 1$ for $1\le i\le m-1$, $f(u_{2i-1}) = (2m-1)(2n-1) + 2m + 1 - i$ for $1\le i\le m$ and $f(u_iv_j) = a_{i,j}$, $1\le i\le 2m-1, 1\le j\le 2n-1$, where $a_{i,j}$ is the $(i,j)$-entry of a $(2m-1,2n-1)$-magic rectangle with constant row sum $(2n-1)(2mn-m-n+1)$ and constant column sum $(2m-1)(2mn-m-n+1)$. One may check that $f$ is a bijection with $g_1= f^+(v_j) = (2m-1)(2mn-m-n+1)$ for $1\le j\le 2n-1$, $g_2=f^+(u_{2i}) = (2n-1)(2mn-m-n+1) + 2(2m-1)(2n-1) + 2m + 2$ for $1\le i\le m-1$, $g_3=f^+(u_{2i-1}) = (2n-1)(2mn-m-n+1) + 2(2m-1)(2n-1) + 2m + 1$ for $1\le i\le m$ and $g_4=f^+(u_{2m}) = 2(2m-1)(2n-1) + m + 2$. It is routine to verify that $g_1, g_2, g_3, g_4$ are distinct. Thus, $\chi_{la}(G(2m, 2n-1)) \le 4$. The theorem holds.\end{proof}

\begin{example}\label{eg-m56}{\rm The following are labelings that give $\chi_{la}(G(5,2))=\chi_{la}(G(6,2))=\chi_{la}(G(6,3))=3$.\\[2mm]
\centerline{
\epsfig{file=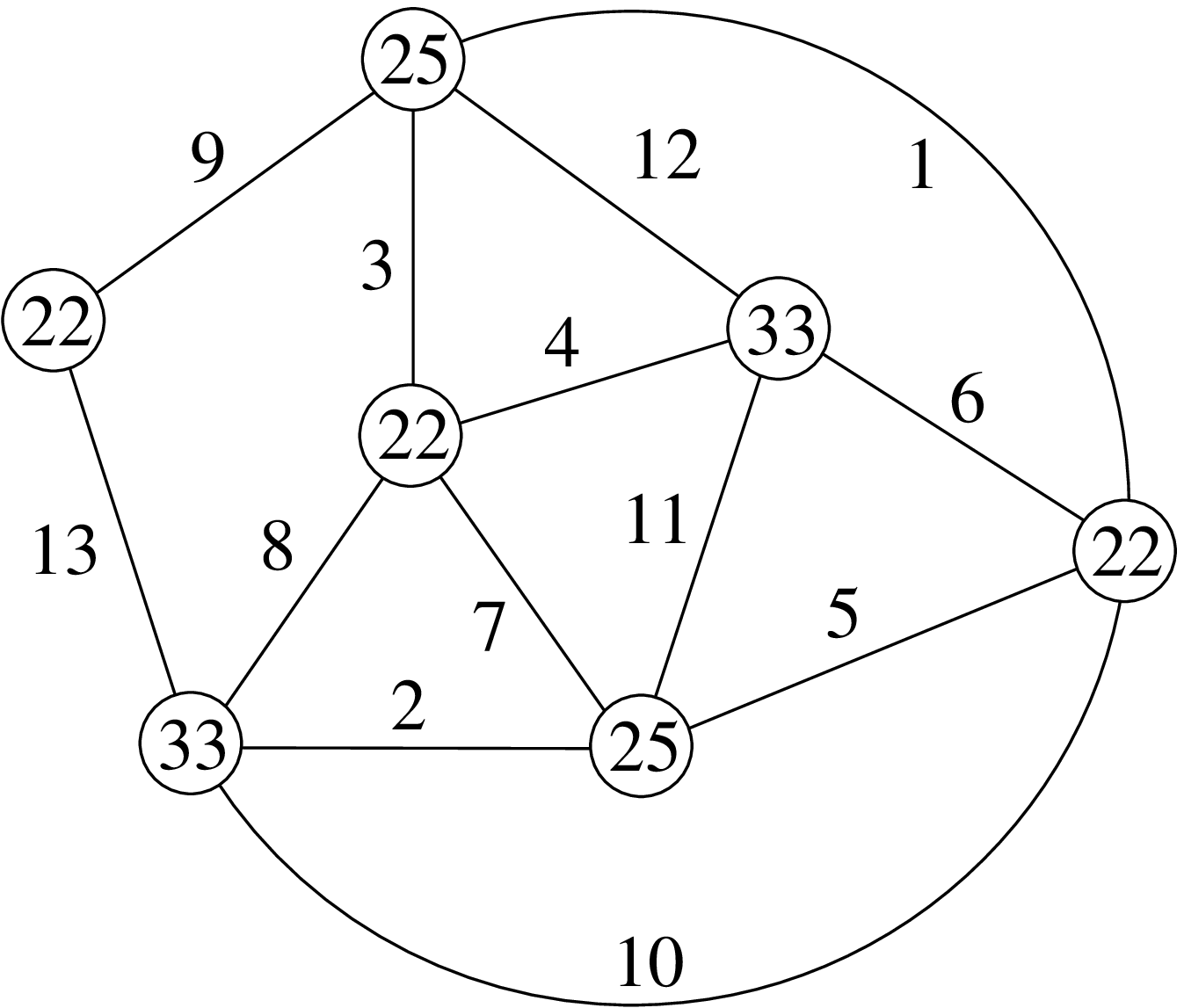, width=3.5cm}\quad \epsfig{file=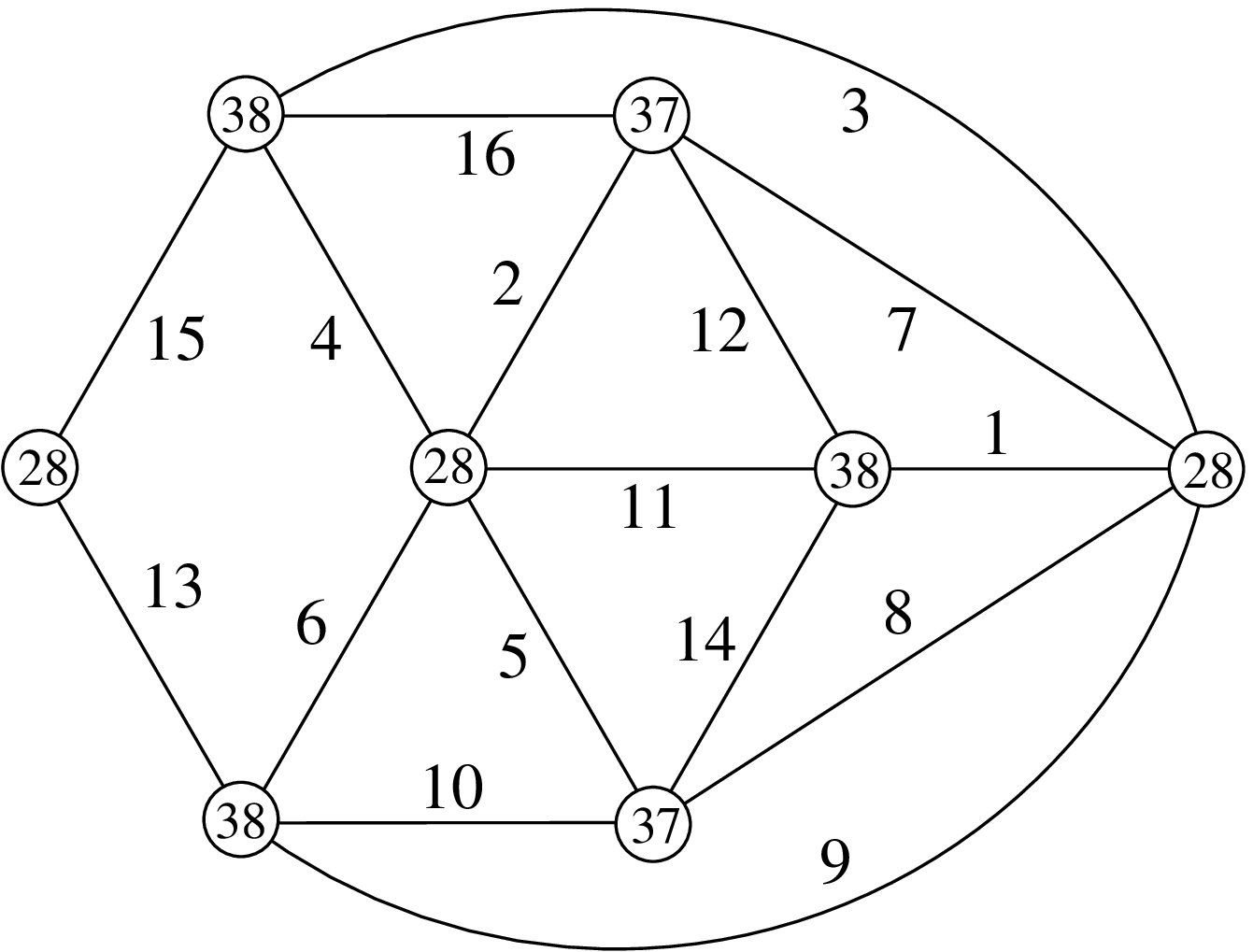, width=4cm} \quad
\epsfig{file=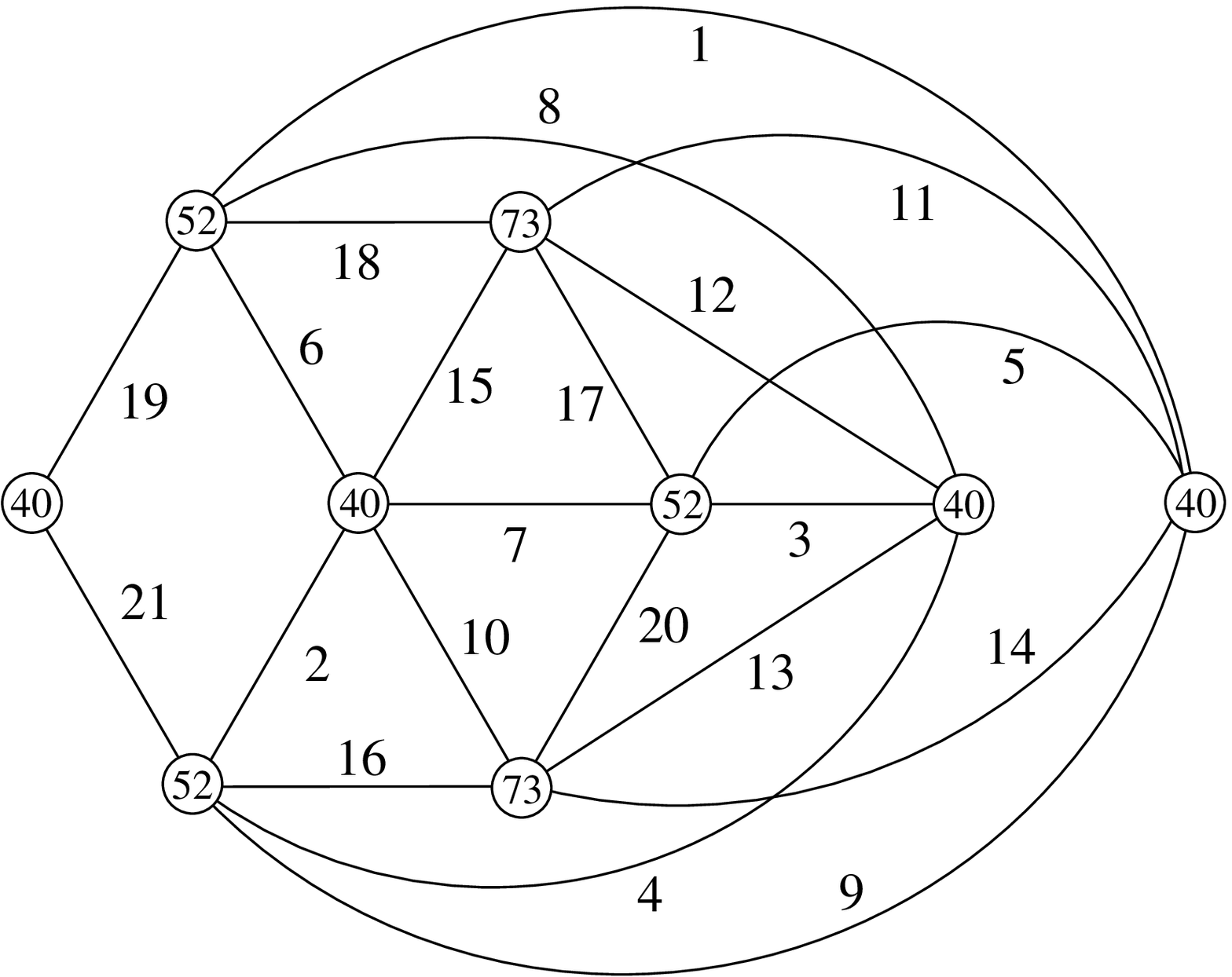, width=5cm}}

\ms\nt Note that $G(5,2)$ and $G(6,2)$ are two graphs we have not considered before.
}
\end{example}
\begin{problem} For $m\ge 5$, find $\chi_{la}(G(m,n))$ for $G(m,n)$ not a graph in Theorem~\ref{thm-Gmevennodd} and Example~\ref{eg-m56}. \end{problem}

\ms\nt Little is known on bipartite graph $G$ with $\chi_{la}(G)=2$ (see~\cite[Theorems 2.11 and 2.12]{Arumugam}). For $m\ge 2, i\ge 1$, let $B(n_1,n_2,\ldots,n_m)$ be the union of $K_{2,n_i}$ with bipartition $(X_i,Y_i)$, where $X_i=\{x_{i-1},x_i\}$, $Y_i=\{y_{i,1},y_{i,2},\ldots, y_{i,n_i}\}$ and $x_m=x_0$.

\ms\nt It is known from \cite[Theorem~2.8 and Theorem~2.12]{Arumugam} that $\chi_{la}(B(1^{[m]})) = \chi_{la}(C_{2m}) =3$ and  $\chi_{la}(B(n^{[2]}))=\chi_{la}(K_{2,2n})=2$ for $n\ge 2$. The following theorem gives another family of bipartite graphs with $\chi_{la}$ equal to 2.

\begin{theorem}\label{thm-Bnm} Suppose $m\ge 3$ and $n\ge 2$. We have $\chi_{la}(B(n^{[m]})) = 2$ if $n$ is even or both $m$ and $n$ are odd; $2\le \chi_{la}(B(n^{[m]}))\le 3$ for odd $n$ and even $m$.
 \end{theorem}

\begin{proof}
First note that the edges in each $K_{2,n}$ are $x_{i-1}y_{i,j}$ and $x_iy_{i,j}$ for $1\le i\le m, 1\le j\le n$.

\nt Suppose $n\ge 2$ is even. Define a bijection $f : E(G) \to [1, 2mn]$ such that
\begin{align*}
f(x_{i-1}y_{i,j}) & =\begin{cases}
(i-1)n+j & \mbox{ for odd }j\in [1, n-1];\\
(2m-i+1)n-j+1 & \mbox{ for even }j\in [2, n],
\end{cases}\\
f(x_{i}y_{i,j}) & =\begin{cases}
(2m-i+1)n-(j-1) & \mbox{ for odd }j\in [1, n-1];\\
(i-1)n+j & \mbox{ for even }j\in [2, n],\end{cases}
\end{align*}
where $1\le i\le m$.

\ms
\nt Recall that $x_m=x_0$. It is easy to verify that $f^+(y_{i,j}) = 2mn+1$ and $f^+(x_i)=2mn^2 + n$ for $1\le i\le m, 1\le j\le n$. Hence, $\chi_{la}(G) \le 2$. Since $\chi_{la}(G) \ge \chi(G) = 2$, we have $\chi_{la}(G) = 2$ for even $n\ge 2$.  \\

\nt Suppose $n$ is odd and $m$ is odd. Let $A$ be a magic $(m,n)$ rectangle.
For $1\le i\le m$, let $(f(x_{i}y_{i,1}), \dots, f(x_{i}, y_{i, n}))$ be the $i$-th row of $A$ and let $f(x_{i-1} y_{i, j})=2mn+1-f(x_{i} y_{i, j})$ for $1\le j\le n$. Clearly $f^+(y_{i,j})=2mn+1$ for $1\le i\le m$ and $1\le j\le n$. Since the row sum of $A$ is $n(mn+1)/2$, $f^+(x_i)=n(2mn+1)$ for $1\le i\le m$. Here, $\chi_{la}(G)\le 2$ and hence $\chi_{la}(G)= 2$.

\ms\nt Suppose $n$ is odd and $m$ is even.
Define a bijection $f : E(G) \to [1, 2mn]$
\begin{align*}
f(x_{i-1}y_{i,j}) & =
(i-1)n+j; \\
f(x_{i}y_{i,j}) &=
(2m-i+1)n-j+1,
\end{align*}
where $1\le i\le m$.

\nt It is easy to verify that $f^+(y_{i,j}) = 2mn + 1$, $f^+(x_0) = n(mn + n + 1)$ and $f^+(x_i) = n(2mn + n + 1)$ for $1\le i\le m-1$. Thus, $\chi_{la}(G)\le 3$.
 \end{proof}
\begin{example}{\rm The following is a local antimagic labeling according to the construction described in the proof above, which induces a $2$-coloring for $B(3^{[3]})$.
\[A=\begin{pmatrix}
2 & 7 & 6\\
9 & 5 & 1\\
4 & 3 & 8\end{pmatrix}.\]
\[\begin{array}{c|*{3}c|*{3}c|*{3}c|}
 & y_{1,1} & y_{1,2} & y_{1,3} & y_{2,1} & y_{2,2} & y_{2,3} & y_{3,1} & y_{3,2} & y_{3,3} \\\hline
x_1 & 2 & 7 & 6 & 10 & 14 & 18 &  * & * & * \\
x_2 & * & * & * & 9 & 5 & 1 & 15 & 16 & 11 \\
x_3 & 17 & 12 & 13 &  * & * & * &  4 & 3 & 8\\\hline
\end{array}\]
It is clear that each row sum is 57 and each column sum is 19.
}
\end{example}

\begin{example}{\rm The following is a local antimagic labeling inducing a $2$-coloring for $B(3^{[4]})$.
\[\begin{array}{c|*{3}c|*{3}c|*{3}c|*{3}c|}
 & y_{1,1} & y_{1,2} & y_{1,3} & y_{2,1} & y_{2,2} & y_{2,3} & y_{3,1} & y_{3,2} & y_{3,3} & y_{4,1} & y_{4,2} & y_{4,3}\\\hline
x_1 & 1 & 5 & 17 & 23 & 19 & 10 &  * & * & * & * & * & *\\
x_2 & * & * & * & 2 & 6 & 15 & 22 & 16 & 14 & * & * & * \\
x_3 & * & * & * & * & * & * &  3 & 9 & 11 & 21 & 18 & 13 \\
x_4 & 24 & 20 & 8 & * & * & * & * & * & * & 4 & 7 & 12\\\hline
\end{array}\]
It is easy to see that the row sum is always 75 and the column sum is always 25.}
\end{example}


\begin{problem} Determine $\chi_{la}(B(n_1,n_2,\ldots,n_m))$ for  $B(n_1,n_2,\ldots,n_m) \not= B(n^{[m]})$.  \end{problem}   

\section{Potential Application}


\nt Suppose $G$ is a connected graph of $p$ vertices and $q$ edges. Assume that each vertex represents an object of certain (possibly identical) weight and each object must be connected to at least one other object by a certain {\it unique} number of uniform bar(s) ranging from 1 to $q$ such that the total support receives by a vertex is the total number of connecting bar(s). Further assume that under safety consideration, no two adjacent objects can receive the same number of support(s) and that under cost saving consideration, we must minimize the number of distinct supports given to all the objects. Thus, this is answered by finding the local antimagic chromatic number of the underlying graph $G$.  



\end{document}